\documentclass[11pt]{amsart}

\usepackage{amsthm,amsmath,amssymb,geometry,listings,multirow,enumerate,textcomp,array}
\usepackage{tikz,tkz-graph,tkz-berge}
\usepackage{color,url,soul,seqsplit,longtable} 
\theoremstyle{definition} \newtheorem{cor}{Corollary}[section]
\theoremstyle{definition} \newtheorem{lem}[cor]{Lemma} 
\theoremstyle{definition} \newtheorem{prop}[cor]{Proposition} 
\theoremstyle{definition} \newtheorem{thm}[cor]{Theorem}
\theoremstyle{definition} 
\theoremstyle{definition} 
\theoremstyle{definition}

\theoremstyle{definition} \newtheorem{defn}[cor]{Definition}
\theoremstyle{definition} 
\theoremstyle{definition} 
\theoremstyle{definition}  
\theoremstyle{definition} 
\theoremstyle{definition}  
\theoremstyle{definition}  
\theoremstyle{definition}  
\theoremstyle{definition} 
\theoremstyle{definition} 
\theoremstyle{definition} 
\theoremstyle{definition} 
\theoremstyle{definition} 
\theoremstyle{definition} \newtheorem{conj}[cor]{Conjecture}

\newcommand{\Z}{\mathbb Z}
\newcommand{\N}{\mathbb N}

\newcommand{\up}{\textup}
\DeclareMathOperator{\PCP}{PCP}

\DeclareMathOperator{\children}{Children}

\renewcommand{\tt}{\texttt}
\newcommand{\LM}{\up{LM}}
\newcommand{\rref}{\up{rref}}

\author{Martin E.\ Malandro and Ken W.\ Smith}

\begin{document}


\title{Partial Difference Sets in $C_{2^n} \times C_{2^n}$}
\address{Department of Mathematics and Statistics, Box 2206, Sam Houston State University, Huntsville, TX 77341-2206}
\email{malandro@shsu.edu, kenwsmith54@gmail.com}

\begin{abstract}
{
We give an algorithm for enumerating the regular nontrivial partial difference sets (PDS) in the group $G_n = C_{2^n}\times C_{2^n}$. 
We use our algorithm to obtain all of these PDS in $G_n$ for $2\leq n\leq 9$, and we obtain partial results for $n=10$ and $n=11$.
Most of these PDS are new.
For $n\le 4$ we also identify group-inequivalent PDS.
Our approach involves constructing tree diagrams and canonical colorings of these diagrams. 
Both the total number and the number of group-inequivalent PDS in $G_n$ appear to grow super-exponentially in $n$.
For $n=9$, a typical canonical coloring represents in excess of $10^{146}$ group-inequivalent PDS, and
there are precisely $2^{520}$ reversible Hadamard difference sets. 
}
\end{abstract}



\maketitle


\section{Introduction}
\label{SecIntro}


The study of partial difference sets in finite groups lies in the intersection of the study of group actions and combinatorial structures. 
In this paper we study the partial difference sets in the group $G_n = C_{2^n} \times C_{2^n}$, where $C_n$ denotes the cyclic group of order $n$. 
While the enumeration of partial difference sets in some classes of groups can be achieved with theorems alone---for instance, a complete characterization of the partial difference sets in $C_{p^n}\times C_{p^n}$ is achieved for all {\em odd} primes $p$ in \cite{LeifmanMuzychuk95}---the enumeration of the partial difference sets in $G_n$ appears to be harder, and seems to require an algorithm combining algebraic conditions with a combinatorial search. 
The main goals of this paper are to describe such an algorithm and to discuss its output for $n\leq 11$. 
In particular, we describe all partial difference sets in $G_n$ for $n\leq 9$, and we provide partial results for $n=10, 11$. The vast majority of these partial difference sets are new.

We begin by providing background information and motivation for this project in this section, and we also discuss the partial difference sets we found in $G_n$ for $n\leq 11$. 
In Section \ref{SecOrbitTree} we discuss the characterization of partial difference sets in $G_n$ in terms of tree diagrams that we use throughout the paper. 
In Section \ref{SecCanonicalColorings} we define canonical colorings of these diagrams, and in Section \ref{SecLM} we recall a result from \cite{LeifmanMuzychuk95} that states what rows $0$, $1$, and $n$ of any canonical coloring representing a PDS in $G_n$ must look like. 
Our algorithm, detailed in Section \ref{SecAlgorithm}, fills in the rest of the rows. 
In Section \ref{SecExamples} we work out a few small examples by hand. 
Finally, in Appendix \ref{SecOutput} we record a representative for every regular nontrivial partial difference set in $G_n$ for all $2\leq n\leq 9$.

\subsection{Strongly Regular Graphs}
A strongly regular graph with parameters $(v,k,\lambda, \mu)$ is a (loopless) graph on $v$ vertices, with each vertex of degree $k$, such that the number of common neighbors of two distinct vertices $x$ and $y$ depends only on whether $x$ and $y$ are adjacent or not; it is $\lambda$ in the first case and $\mu$ otherwise. 
The complement of a graph exchanges edges with non-edges.
The complement of a strongly regular graph is strongly regular.

The study of strongly regular graphs dates back to Bose \cite{Bose1963}. 
Research into strongly regular graphs is a rich vein within algebraic combinatorics, involving graph theory, linear algebra, group theory and algebraic number theory.
As Peter Cameron explains,
``Strongly regular graphs stand on the cusp between the random and the
highly structured"%
(\cite{BeinekeWilson2004}, p. 204.)

The adjacency matrix of a $(v,k,\lambda, \mu)$ strongly regular graph has three eigenvalues: the degree $k$, and the two roots of the quadratic polynomial 
$x^2 - (\lambda-\mu)x-(k-\mu).$
If the graph is disconnected then $\mu=0$ and the graph is simply the union of complete graphs.
If $\mu=k$ then any two nonadjacent vertices have the same neighborhood and so any two adjacent vertices have no neighbors in common. 
This is the complement of the case $\mu=0.$
We consider the cases where $\mu =0$ or $\mu=k$ to be trivial and henceforth require that $\mu$ lie strictly between 0 and $k$.
This implies that zero is not an eigenvalue of the graph and so the adjacency matrix of the graph is invertible.

A graph in which every vertex has degree $k$ and where the adjacency matrix has three distinct eigenvalues (one of which must be $k$) is strongly regular.
A positive latin square graph (or ``pseudo-latin square graph") is a strongly regular graph $PL_t(m)$ with parameters
$(m^2, t(m-1), t^2-3t+m, t(t-1))$
and eigenvalues $k, m-t, -t$.
A negative latin square graph is a SRG $NL_t(m)$ with parameters 
$(m^2, t(m+1), t^2+3t-m, t(t+1))$; the eigenvalues of its adjacency matrix are $k, t, t-m$.
(The parameters of $NL_t(m)$ can be obtained from those of $PL_t(m)$ by replacing $t$ and $m$ by $-t$ and $-m$, respectively---see \cite{BeinekeWilson2004}, p. 207.)

Good sources on strongly regular graphs include the monographs \cite{GodsilRoyle2001} by Godsil and Royle, \cite{Biggs1993} by Norman Biggs and 
 \cite{BCN1989} by Brouwer, Cohen, and Neumaier.
Chapter 2 of Cameron and Van Lint's text \cite{CameronvanLint1991} also provides a very readable introduction.
Chapter 21 of \cite{vanLintWilson2001} has a section on strongly regular graphs, while Chapter 17 has a section on Latin Squares.

\subsection{Partial Difference Sets} 
Partial difference sets are the main objects of study in this paper.

\begin{defn}\label{DefPDS}
A  $(v,k,\lambda, \mu)$ {\em partial difference set} (PDS; plural also PDS) in a group $G$ of order $v$ is a set $S\subseteq G$ of size $k$ such that the multi-set 
$\{ s_1s_2^{-1} : s_1, s_2 \in S\}$ has $\lambda$ or $\mu$ occurrences of each nonidentity element $g$, depending on whether or not $g \in S,$ respectively.
A partial difference set $S$ is {\em regular} if it does not contain the identity element.
\end{defn}

Let $S$ be a partial difference set in a group $G$.
The regular representation of $G$ maps $S$ to a (0,1)-matrix that is an adjacency matrix of the associated Cayley graph.
If $G$ is abelian, the eigenvalues of that matrix are exactly the set of character values of $S$ 
$\{\chi(S)=\sum_{s\in S}\chi(s) : \chi\in G_n^*\},$ also called the {\em eigenvalues} of $S$.
The {\em nontrivial eigenvalues} of $S$ are the eigenvalues of $S$, except for the eigenvalue arising from the trivial character $\chi_0$, which maps $S$ to $k$.


A partial difference set is {\em regular} if it does not contain the identity element. 
If $S$ is a partial difference set and $1 \in S$ then removing $1$ from $S$ yields a regular PDS.  
This new PDS has eigenvalues 
one less than those of the original PDS.
  Conversely one may add the identity to a regular PDS to obtain a PDS with eigenvalues one more than the original PDS.
For this reason we will focus on regular PDS and assume $1 \not\in S.$
 Henceforth whenever we speak of a PDS we shall always mean regular nontrivial PDS.


If $S$ is a PDS in an abelian group $G$ then the nontrivial characters each map $S$ to one of two eigenvalues and so we may partition the nonidentity members of the dual group into two sets, depending on the value the characters assign to $S$.  
Each of these subsets of $G^*$ forms a PDS in $G^*$.
Such a ``dual PDS"  need not have the same parameters as the original PDS.

The Cayley graph $\Gamma_G(S)$ of a regular PDS $S$ is strongly regular.
Conversely, if a strongly regular graph has an automorphism group $G$ acting sharply transitively on the vertices then we may choose a vertex to label with the identity of $G$ and the neighbors of that vertex will form a PDS in $G$.

The complement $S^c$ of a PDS $S$ is the set of all nonidentity elements of $G$ not in $S$.
Note that if $|S|=k$ then $|S^c|=v-k-1$. 
 We define our complements in this way so that the complement of a regular PDS $S$ is also a regular PDS, and so that the complement of $\Gamma_G(S)$ is equal to $\Gamma_G(S^c)$.

The study of partial difference sets was initiated by Ma in \cite{Ma1984}, although strongly regular Cayley graphs appear in earlier work by Bridges and Mena \cite{BridgesMena1979}.
Later articles on PDS, including
\cite{LeifmanMuzychuk95}, \cite{Ma1994} and \cite{Ma1997},
demonstrate that outside of the Paley $(4n+1,2n,n-1,n )$ parameters, PDS do not exist in cyclic groups, and in abelian groups of rank two they occur rarely except in the family
 $C_{p^n} \times C_{p^n},$ where $C_{p^n}$ is the cyclic group of order $p^n$ and $p$ is prime. 
Theorem 5.6(b) in \cite{Ma1994} shows that regular nontrivial PDS do not exist in groups with Sylow $p$-subgroup $C_{p^a}\times C_{p^b}$ where $a \ne b.$
Thus in rank two abelian groups we should focus on $C_{p^n}\times C_{p^n}$.

The main result from \cite{LeifmanMuzychuk95} is a complete characterization of the PDS in $C_{p^n} \times C_{p^n}$ in the case that $p$ is an odd prime. 
The case studied in the present paper, $C_{2^n} \times C_{2^n}$, appears to be significantly more complicated, although \cite{LeifmanMuzychuk95} contains partial results for this case upon which we build in this paper.
An early investigation in PDS in $C_{2^n} \times C_{2^n}$ appears in the masters thesis of Stuckey and Gearheart (\cite{StuckeyThesis}, \cite{GearheartThesis}).

The latin square parameters provides a setting in which the dual PDS has the same parameters as the original PDS. Given a PDS $S$ with the parameters of $PL_t(m)$, define
 $S^*:= \{ \chi \in G^*: \chi(S) = m-t \}$. Then $S^*$ is a PDS in $G^*$ with the same parameters as $S$. If $S$ has the parameters of $NL_t(m)$, then the dual set
 $S^*:= \{ \chi \in G^*: \chi(S) = t-m \}$ is a PDS in $G^*$ with the same parameters as $S$.
 
\subsection{Partial Congruence Partitions}
Suppose a group of order $m^2$ has $t$ subgroups $H_1, H_2, ..., H_t$, each of order $m$, each pair meeting only in the the identity.  
Then the set $S$ consisting of the union of the subgroups, less the identity, gives a strongly regular Cayley graph 
with latin square parameters $PL_t(m)$.
If $t=1$ then the graph is trivial; it is a disjoint union of complete graphs.
If our group is $G_n=C_{2^n} \times C_{2^n}$ then we may create such graphs for $t=1$, $t=2$, and $t=3$. We denote these classes of graphs by $\PCP(1)$, $\PCP(2)$, and $\PCP(3)$.
If $t=2$ and $n > 2$, this construction gives the unique graph (up to group isomorphism) with those parameters, and the number of distinct PDS with this graph is $3 \cdot 2^{2n-2}$.
If $t=3$ and $n > 3$, this construction gives the unique graph (up to group isomorphism) with those parameters, and the number of distinct PDS with this graph is $2^{3n-3}$. (See the PCP(2) and PCP(3)$^c$ entries of Tables \ref{TableSummary2}--\ref{TableSummary3}.)

\subsection{Reversible Hadamard Difference Sets}

A Hadamard difference set (HDS) is a set $D$ of size $2m^2-m$ in a group $G$ of order $4m^2$ such that the multi-set
$\{ s_1s_2^{-1} : s_1, s_2 \in S\}$ has $\lambda=m^2-m$ occurrences of every nonidentity element $g\in G$.
A difference set $D$ is {\em reversible} if $D$ is equal to $D^{(-1)}$, its sets of inverses.

Let $D\subseteq G$ and let $1$ denote the identity of $G$. 
Then $D$ is a reversible HDS with $1 \in D$ if and only if $D^c$ is a PDS with degree $2m^2+m.$ 
Also $D$ is a reversible HDS with $1 \not \in D$ if and only if $D$ is a PDS with degree $2m^2-m.$ 
Note that for $G=G_n=C_{2^n}\times C_{2^n}$ we have $m = 2^{n-1}$ and hence in this case the PDS giving rise to reversible HDS have order 
$2m^2-m = 2^{2n-1}-2^{n-1}$ and 
$2m^2+m = 2^{2n-1}+2^{n-1}$. 
We will call such PDS in $G_n$ {\em reversible Hadamard partial difference sets (RHPDS) of type A and type B}, respectively, even though the RHPDS of type B are not actually themselves reversible HDS.

In general, a reversible Hadamard difference set with degree $m^2-m$ is a PDS with the parameters of $PL_{m-1}(m)$;
if the degree is  $m^2+m$ then the PDS has the parameters of $NL_{m+1}(m)$.
A good reference for Hadamard difference sets is \cite{DavisJedwab1993}.



\subsection{Summary of Our Results} 
\label{SecIntroSummary}
Our main results are our algorithm for enumerating the PDS in $G_n = C_{2^n} \times C_{2^n}$ (as presented in Section \ref{SecAlgorithm}) and its output for $n\leq 11$, which we discuss now and also in Appendix \ref{SecOutput}.

As discussed after Definition \ref{DefPDS}, we seek to enumerate only regular nontrivial PDS. 
Furthermore, since $|G_n|$ is even, $S\subseteq G_n$ is a PDS of even size if and only if $S^c$ is a PDS of odd size. 
Hence it suffices to enumerate only the PDS of a single parity. 
We choose to enumerate those of even size. 

It turns out that every PDS in $G_n$ is subset of the orbit tree of $G_n$ (Definition \ref{DefOrbitTree}). 
Furthermore any automorphism of the rooted orbit tree applied to a PDS is also a PDS, and hence we furthermore only need to enumerate one representative of each tree automorphism class of partial difference sets. 
We call the representatives that we enumerate {\em canonical} partial difference sets (Definition \ref{defCanonical}). 
A summary of the number of PDS in $G_n$ for $n\leq 11$ may be found in Table \ref{TableSummary}. 
This summary is broken down further in Tables \ref{TableSummary2}--\ref{TableSummary3}, and further still (into individual PDS) in Appendix \ref{SecOutput}.

The PDS in $G_n$ fall into three categories: The RHPDS, the partial difference sets of type PCP, and ``other", which we call {\em sporadic}. 
For $n\leq 9$ our enumeration of the (regular nontrivial) canonical PDS of even size is complete. 
For $n=10,11$ we halted our implementation of our algorithm before it was finished so our results for these values of $n$ are likely incomplete. 
We suspect that for $n=10$ our data is missing 16 canonical RHPDS of type B, while for $n=11$ it is missing 32 canonical RHPDS and one sporadic canonical PDS. 
All of the canonical PDS output by our program (including the complete list of canonical PDS of even size in $G_n$ for $n\leq 9$) are available at our website \cite{OurWebsite}.

Next we discuss the total number (i.e., not up to any kind of isomorphism) of regular nontrivial PDS in $G_n$ of even size, recorded in the ``\#total" columns of Tables \ref{TableSummary}--\ref{TableSummary3}. 
We computed these numbers from our canonical PDS with a straightforward orbit-stabilizer approach. 
The most interesting family of PDS in $G_n$ are the RHPDS, which appear to constitute the vast majority of the PDS in $G_n$---See Tables \ref{TableSummary2}--\ref{TableSummary3}.

The results of this project suggested the following conjecture
\footnote{James A.\ Davis and John B.\ Polhill asked for the total number of reversible Hadamard difference sets, leading to our conjecture.}:

\begin{conj}The total number of reversible Hadamard difference sets in $G_n = C_{2^n}\times C_{2^n}$ of:
\begin{itemize}
	\item type A is $(2^{2^n-2})(2^n+1)$, and
	\item type B is $(2^{2^n-2})(2^n-1)$.
\end{itemize}
Hence the total number of reversible Hadamard difference sets in $G_n$ is equal to $2^{2^n+n-1}$.
\end{conj}

Unfortunately, space constraints prevent us from displaying the evidence for this conjecture beyond $n=5$ in our tables. 
Nevertheless all of our computations for which we have complete results (i.e., $n\leq 9$) are of course consistent with this conjecture. 
Note that this conjecture suggests that there is a process for determining a reversible HDS in $G_n$ that consists of answering a sequence ${2^n+n-1}$ yes/no questions, and that none of these questions is ``Do you want your reversible HDS to be of type A or type B?" (In other words, none of these questions is ``Do you want the group identity to be in the HDS or not?")


Next we discuss the numbers in the ``\#up to $G_n$-equivalence" column of Table \ref{TableSummary}. 
Two PDS $S, S'\subseteq G_n$ are {\em $G_n$-equivalent} if there exists an automorphism $\sigma$ of $G_n$ for which $\sigma(S) = S'$. 
While it might seem natural to want to enumerate the PDS in $G_n$ up to $G_n$-equivalence, our results suggest that it is the notion of orbit tree equivalence, {\em not} the notion of $G_n$-equivalence, that is most natural for enumerating the PDS in $G_n$. 
The orbit tree of $G_n$, viewed as a rooted tree, has an automorphism group of order $3 \cdot 2^{3(2^{n-1}-1)+1}$. 
This dwarfs the order of the automorphism group of $G_n$ itself, which is only $3\cdot 2^{4n-3}$.
(See Lemma \ref{Automorphisms}.)
 An automorphism of $G_n$ induces an automorphism of the orbit tree, so PDS which are $G_n$-equivalent are also orbit tree equivalent. 
 However the converse is certainly not true in general based on order considerations alone, and for $n\geq 4$ there are many orbit tree equivalent PDS which are not $G_n$-equivalent. 

For $n\leq 4$ we computed these numbers with \texttt{GAP} \cite{GAP4}. 
For $n=5$ we did not attempt to sort the $\approx 83$ billion distinct PDS of even size into $G_n$-equivalence classes. 
For $n=5$ and higher, to obtain these numbers we simply divided the total number of PDS in $G_n$ by the order of the automorphism group of $G_n$.
 Despite the (likely dramatic) underestimate of the number of $G_n$-inequivalent PDS this produces, it is enough to demonstrate the importance of enumerating up to orbit tree equivalence rather than $G_n$-equivalence. 
 For $n=9$, for instance, it would be impossible to store a representative of each of the $\geq 1.5\times 10^{146}$ $G_9$-equivalence classes of PDS (as there are only $\approx 10^{80}$ atoms in the observable universe), whereas we can store a representative of each of the 35 orbit tree equivalence classes of PDS in $G_9$ in just a few hundred kilobytes.

Finally we mention that each PDS provides us with a strongly regular Cayley graph, and we have not considered the problem of how many PDS there are in $G_n$ up to {\em graph} isomorphism.
 We do not know how graph isomorphism classes relate to the tree and group equivalence classes but it is possible that orbit tree equivalence is the same as graph isomorphism.

%

\setlength\extrarowheight{3pt}
\begin{table}%
\begin{tabular}{|c|c|c|c|}
\hline
$n$&\#canonical&\#up to $G_n$-equivalence &\#total\\
\hline
2& 3& 3 & 32\\
\hline
3& 8& 8 & $1392$ \\
\hline
4& 13& 268 & $709312$\\
\hline
5& 12& $\geq 210262$ & $82678125312 \approx 8.3\times 10^{10}$ \\
\hline
6& 19& $\geq 1.2\times 10^{14}$ & $\approx 7.7\times 10^{20}$\\ 
\hline
7& 19& $\geq 2.5\times 10^{32}$ & $\approx 2.6\times 10^{40}$ \\ 
\hline
8& 35& $\geq 1.3\times 10^{70}$ & $\approx 2.1\times 10^{79}$ \\ 
\hline
9& 35& $\geq 1.5\times 10^{146}$ & $\approx4.1\times 10^{156}$\\
\hline
10& $\geq 50$ & $\geq 2.9\times 10^{299}$ & $\geq 1.2\times 10^{311}$ \\
\hline
11& $\geq 32$ & $\geq 1.2\times 10^{606} $ & $\geq 8.2\times 10^{619}$ \\
\hline
\end{tabular}
\caption{The regular nontrivial PDS in $C_{2^n} \times C_{2^n}$ of even size}
\label{TableSummary}
\end{table}

\setlength\extrarowheight{3pt}
\begin{table}%
\begin{tabular}{|c|c|c|c|c|c|}
\hline
$n$&$k$&type&\#canonical&\#total&$\approx$ \% of total in $G_n$\\
\hline
\hline
\multirow{2}{*}2
&6&RHPDSA, PCP(2) & 2 & 20 & 62.5\% \\
& 10 & RHPDSB, PCP(3)$^c$ & 1 & 12 & 37.5\% \\
\hline
\hline
\multirow{5}{*}3
& 14 & PCP(2) & 1 & 48 & 3.448\% \\
& 18 & Sporadic & 1 & 64 & 4.598\% \\
& 28 & RHPDSA & 2 & 576 & 41.38\% \\
& 36 & RHPDSB & 2 & 448 & 32.18\% \\
& 42 & PCP(3)$^c$ & 2 & 256 & 18.39\% \\
\hline
\hline
\multirow{7}{*}4
& 30 & PCP(2) & 1 & 192 & 0.02707\% \\
& 90 & Sporadic & 2 & 81920 & 11.55\% \\
& 120 & RHPDSA & 4 & 278528 & 39.27\% \\
& 136 & RHPDSB & 3 & 245760 & 34.65\% \\
& 150 & Sporadic & 1 & 98304 & 13.86\% \\
& 180 & Sporadic & 1 & 4096 & 0.5775\% \\
& 210 & PCP(3)$^c$ & 1 & 512 & 0.07218\% \\
\hline
\hline
\multirow{6}{*}5
& 62 & PCP(2) & 1 & 768 & $9.289 \times 10^{-7}$\% \\
& 372 & Sporadic & 1 & 1073741824 & 1.299\% \\
& 496 & RHPDSA & 4 & 35433480192 & 42.86\% \\
& 528 & RHPDSB & 4 & 33285996544 & 40.26\% \\
& 558 & Sporadic & 1 & 12884901888 & 15.58\% \\
& 930 & PCP(3)$^c$ & 1 & 4096 & $4.954 \times 10^{-6}$\% \\
\hline
\hline
\multirow{6}{*}6
& 126 & PCP(2) & 1 & 3072 & $3.965 \times 10^{-16}$\% \\
& 1890 & Sporadic & 1 & $\approx 7.4\times 10^{19}$ & 9.524\% \\
& 2016 & RHPDSA & 8 & $\approx 3.0\times 10^{20}$ & 38.69\% \\
& 2080 & RHPDSB & 7 & $\approx 2.9 \times 10^{20}$ & 37.50\% \\
& 2142 & Sporadic & 1 & $\approx 1.1\times 10^{20}$ & 14.29\% \\
& 3906 & PCP(3)$^c$ & 1 & 32768 & $4.229 \times 10^{-15}$\% \\
\hline
\end{tabular}
\caption{The regular nontrivial PDS in $C_{2^n} \times C_{2^n}$ of even size, by size ($k$), $n=2,\ldots, 6$}
\label{TableSummary2}
\end{table}

\setlength\extrarowheight{3pt}
\begin{table}%
\begin{tabular}{|c|c|c|c|c|c|}
\hline
$n$&$k$&type&\#canonical&\#total&$\approx$ \% of total in $G_n$\\
\hline
\hline
\multirow{5}{*}7
& 254 & PCP(2) & 1 & 12288 & $4.751 \times 10^{-35}$\% \\
& 8128 & RHPDSA & 8 & $\approx 1.1 \times 10^{40}$ & 42.43\% \\
& 8256 & RHPDSB & 8 & $\approx 1.1 \times 10^{40}$ & 41.78\% \\
& 8382 & Sporadic & 1 & $\approx 4.1 \times 10^{39}$ & 15.79\% \\
& 16002 & PCP(3)$^c$ & 1 & 262144 & $1.014 \times 10^{-33}$\% \\
\hline
\hline
\multirow{6}{*}8
& 510 & PCP(2) & 1 & 49152 & $2.307 \times 10^{-73}$\% \\
& 32130 & Sporadic & 1 & $\approx 3.7 \times 10^{78}$ & 17.39\% \\
& 32640 & RHPDSA & 16 & $\approx 7.4 \times 10^{78}$ & 34.92\% \\
& 32896 & RHPDSB & 15 & $\approx 7.4 \times 10^{78}$ & 34.65\% \\
& 33150 & Sporadic & 1 & $\approx 2.8 \times 10^{78}$ & 13.04\% \\
& 64770 & PCP(3)$^c$ & 1 & 2097152 & $9.843 \times 10^{-72}$\% \\
\hline
\hline
\multirow{5}{*}9
& 1022 & PCP(2) & 1 & 196608 & $4.824 \times 10^{-150}$\% \\
& 130816 & RHPDSA & 16 & $\approx 1.7 \times 10^{156}$ & 42.19\% \\
& 131328 & RHPDSB & 16 & $\approx 1.7 \times 10^{156}$ & 42.02\% \\
& 131838 & Sporadic & 1 & $\approx 6.4 \times 10^{155}$ & 15.79\% \\
& 260610 & PCP(3)$^c$ & 1 & 16777216 & $4.116 \times 10^{-148}$\% \\
\hline
\hline
\multirow{6}{*}{10}
& 2046 & PCP(2) & $1$ & $786432$ &  \\
& 521730 & Sporadic & $\geq 1$ & $\geq 4.6 \times 10^{310}$ &  \\
& 523776 & RHPDSA & $\geq 32$ & $\geq 4.6 \times 10^{310}$ & \\
& 524800 & RHPDSB & $\geq 15$ & $\geq 1.1 \times 10^{310}$ &  \\
& 525822 & Sporadic & $\geq 1$ & $\geq 1.7 \times 10^{310}$ &  \\
& 1045506 & PCP(3)$^c$ & $1$ &  134217728 & \\ 
\hline
\hline
\multirow{4}{*}{11}
& 4094 & PCP(2) & $1$ &  3145728 & \\ 
& 2096128 & RHPDSA & $\geq 16$ & $\geq 4.1 \times 10^{618}$ & \\
& 2098176 & RHPDSB & $\geq 16$ & $\geq 4.1 \times 10^{618}$ & \\
& 4188162 & PCP(3)$^c$ &  $1$ &  1073741824 & \\ 
\hline
\end{tabular}
\caption{The regular nontrivial PDS in $C_{2^n} \times C_{2^n}$ of even size, by size ($k$), $n=7,\ldots, 11$. Note that this table is likely incomplete for $n=10,11$.}
\label{TableSummary3}
\end{table}

\section{Orbit trees and eigenvalues}
\label{SecOrbitTree}

For $n\in \N$ we take the cyclic group $C_n$ to be $C_n = \left\{\frac{x}{n}:x\in \{0 ,\ldots, n-1\} \right \}$ under the operation of addition mod 1. 
In this way $C_{2^n}$ is a subgroup of $C_{2^{n+1}}$. 
Let $n\geq 2$ and $G_n = C_{2^n}\times C_{2^n}$. 
In this section we discuss the orbit tree $T_n$ of $G_n$, we describe how to compute the eigenvalues of a subset $S$ of $T_n$, and we recall how the eigenvalues of $S$ reveal whether or not $S$ is a PDS. 
Recall that by PDS we will always mean regular nontrivial PDS.




The strongly regular graphs corresponding to our PDS all have integer eigenvalues and since $G_n$ is abelian, a multiplier theorem (Theorem 4.1 in \cite{Ma1994}, p. 228) applies.
This theorem says that a partial difference set $S$ in $G_n$ is fixed by any map
$g \mapsto u\cdot g$ where 
$u$ is an odd integer.
The maps $g = (a,b) \mapsto u\cdot g = (ua, ub)$ ($u$ odd) are called {\em multipliers} and 
$\{ u : u \text{ odd }, 1 \le u \le 2^n\} = U(2^n)$ is the {\em multiplier group} (which is a group under multiplication mod $2^n$).


A directed rooted tree is a tree with a unique root, where every node in the tree has a unique directed path beginning at it and ending at the root.  
We describe a node pointing to another as a {\em child} pointing to a {\em parent}; $\children(X)$ will represent the set of all children of the node $X$.
\begin{defn}
\label{DefOrbitTree}
The {\em orbit tree} $T_n$ of $G_n$ is the directed rooted tree given by the following construction.
\begin{enumerate}
	\item The root, i.e., the unique node at level 0 of $T_n$, is the singleton set consisting of the identity: $\{(0,0)\}$.
	\item The three children of the root, i.e., the three nodes at level 1 of $T_n$,  are the singleton sets consisting of the elements of $G_n$ of order two: $\{(0,\frac{1}{2})\}$, $\{(\frac{1}{2},0)\}$, and $\{(\frac{1}{2},\frac{1}{2})\}$.
	\item For a node $N$ at level $i$ with $1\leq i<n$ of $T_n$, $N$ has two children 
 $C_0$, $C_1$, given by the following construction: Choose any $x\in N$. Then
	\begin{align*}
	C_0 &= \left\{\frac{q}{2}x:q\up{ is odd}, 1\leq q<2^{i+1}\right\},\\
	C_1 &= \{x+t:x\in C_0 \},
	\end{align*}
	where $t$ is either of the elements of $G_n$ of order two not along the path to the root from $N$. (Any choice for $x$ gives the same set $C_0$, and either choice for $t$ gives the same set $C_1$.)
\end{enumerate}
\end{defn}
The orbit tree is so named because the nodes of $T_n$ are the orbits of $G_n$ under the action of the multiplier group $U(2^n)$.
The nodes at level $i$ of $T_n$ form a partition of the elements of $G_n$ of order $2^i$. 
Note that removing level $n$ of $T_n$ gives $T_{n-1}$.
(The orbit tree $T_n$ is explicitly used in  the Masters thesis \cite{GearheartThesis}.
An equivalent and independently discovered version of the orbit tree also appears in \cite{LeifmanMuzychuk95}.)

It will be useful to index the nodes in $T_n$ by ordered pairs.
Let $T_n(0,0)$ denote the root of $T_n$. 
In general $T_n(i,j)$ will denote the node of $T_n$ in the $i$th ``row" and $j$th ``column" of $T_n$. 
Specifically, let $T_n(1,0) = \{(0,\frac{1}{2})\}$, $T_n(1,1) = \{(\frac{1}{2},0)\}$, and $T_n(1,2) = \{(\frac{1}{2},\frac{1}{2})\}$, and given $T_n(i,j)$ with $i\geq 1$, $T_n(i+1,2j)$ and $T_n(i+1,2j+1)$ refer to the children $C_0$ and $C_1$ of $T_n(i,j)$, respectively. We emphasize that $|T_n(i,j)| = 2^{i-1}$ for $i>0$.

The multiplier theorem guarantees that any PDS $S$ in $G_n$ will be a union of elements from the nodes of $T_n$. (This statement may also be found as part of Theorem 1.6 in \cite{LeifmanMuzychuk95}.) Whenever we speak of a subset of $T_n$ we shall always mean a subset of the nodes of $T_n$.
If $S\subseteq T_n$
 we obtain a subset $\bar S\subseteq G_n$ in the obvious way:
$\bar S = \cup_{N\in S} N.$
When $S\subseteq T_n$ we shall regard $S$ a subset of $G_n$ as well, writing $S\subseteq G_n$ instead of $\bar S\subseteq G_n$. 
Similarly, if $X\subseteq G_n$ and there is a subset $S\subseteq T_n$ for which $\bar S = X$, we shall regard $X$ as a subset of $T_n$ as well. In this way we regard partial difference sets in $G_n$ as subsets of $T_n$.

The unique path from a node to the root 
of $T_n$ is a  cyclic subgroup of $G_n$ and so each node may be identified with  a particular cyclic subgroup.



For $S\subseteq T_n$, we define the values $\{\chi(S) : \chi\in G_n^*\}$ to be in agreement with the eigenvalues of $S$ thought of as a subset of $G_n$: 
For $N\in T_n$ and $\chi\in G_n^*$ we set $\chi(N) = \sum_{g\in N} \chi(g)$, and for $S\subseteq T_n$ we set $\chi(S) = \sum_{N\in S}\chi(N)$. The following well-known theorem specifies exactly when $S\subseteq T_n$ is a PDS in terms of the eigenvalues of $S$. 
The statements in this theorem may be found among the results in \cite{LeifmanMuzychuk95}.

\begin{thm}
\label{ThmPDSEig}
$S\subseteq T_n$ is a
 PDS if and only if $S$ has exactly three distinct eigenvalues $k, r, s$, with $k = \chi_0(S)$ being the trivial eigenvalue, $r>0$, $s<0$, $k\equiv r\equiv s\pmod{2^n}$, $s = r-2^n$, and $\chi(S)\in\{r,s\}$ for all nontrivial characters $\chi$. 
\end{thm}

We emphasize that for any $S\subseteq G_n$, $k = |S|$. 
Fortunately, it turns out that we do not have to examine $\chi(S)$ for {\em every} character $\chi\in G_n^*$ to tell whether or not $S\subseteq T_n$ is a PDS. 
In general, as Lemma \ref{LemGalois} below shows, we will only need to check $|T_n|$ many characters.


The nodes of the tree $T_n$ are the orbits of the additive group $G_n$ under the action of $U(G_n)$.
The dual group $G_n^*$ is isomorphic to $G_n$ and so it also may be partitioned into orbits under the action of $U(G^*_n).$ 
The orbit of a character $\chi\in G^*$ is all characters $\chi^m$ where $m$ is any odd integer.

\begin{lem}
\label{LemGalois}
If characters 
$\chi_1, \chi_2\in G_n^*$ are in the same $U(G_n^*)$-equivalence class then for any node $M\in T_n$,  $\chi_1(M)=\chi_2(M)$.
\end{lem}

\begin{proof}
If $\chi_1$ and $\chi_2$ are in the same $U(G_n^*)$-equivalence class then there is an odd integer $m$ such that 
$\chi_2=\chi_1^m$ and so, for any group element $g \in G$, $\chi_2(g)=\chi_1(g^m).$
But the map $g\mapsto g^m$ is a permutation on a node $M$ of $T_n$
and so 
\[ \chi_2(M)=\sum_{g\in M}\chi_2(g) =\sum_{g^m\in M}\chi_1(g) =\sum_{g\in M}\chi_1(g) =\chi_1(M).\]
\end{proof}

For functions $f,g$ on $T_n$, we say $f$ and $g$ {\em agree} on $T_n$ if $f(N)=g(N)$ for all $N\in T_n$.
The above lemma tells us that $U(G_n^*)$-equivalent characters agree on $T_n$.
We now describe explicitly the action of the characters on $T_n$.

Let $N\in T_n$ be a node at level $n$, so $N = T_n(n,j)$ for some $j$. 
Let $P^+$ be the path from $N$ to the root of $T_n$ (including both $N$ and the root), and let 
\[
P^- = \{X\in T_n: \up{$X$ is a child of some $M\in P^+$ and $X\notin P^+$} \}.
\]
Define $\chi_N:T_n \rightarrow \Z$ by, for $M\in T_n$,
\[
\chi_N(M) = 
\begin{cases}
|M| & \up{if $M\in P^+$};\\
-|M| & \up{if $M\in P^-$};\\
0 & \up{otherwise.}
\end{cases}
\]

\begin{thm}
\label{ThmHighestOrderCharDescription}

If $\chi\in G_n^*$ is a character of highest order, there exists a unique $N\in T_n$ at level $n$ so that $\chi$ and $\chi_N$ agree on $T_n$.
\end{thm}

\begin{proof}
A character $\chi\in G_n^*$ of highest order has order $2^n$ and its kernel is a cyclic group of order $2^n$. 
Half of this kernel lies in a single node $N$ at level $n$; the rest of the kernel of $\chi$ lies in the nodes on the path from $N$ to the root $\{(0,0)\}.$
Therefore, for any node on this path, $\chi(M)=|M|$ and so $\chi$ agrees with $\chi_N$ on this path.
We will show that $\chi$ agrees with $\chi_N$ on all nodes in $T_n.$

Any child  $M$ of a node on this path, that is itself not on this path, consists of elements $g \in M$ such that  $2g$
  is in the kernel of $\chi$, yet $g$ is not.
Thus $\chi(g) = -1$ and so $\chi(M) = -|M| = \chi_N(M).$

Finally, those nodes $M$ that are not on the path from $N$ to the root, and that are also not children of nodes on this path, must then consist of group elements $g$ such that $\chi(g)$ is a primitive root of unity of order $2^\ell$ where $\ell \ge 2.$ 
Any such complex $2^\ell$th root of unity $\zeta$ has the property that $\zeta^{2^{\ell-1}+1}=-\zeta$ and so the sum of distinct odd powers of $\zeta$ is zero.  
Therefore, $\chi(M)=0 = \chi_N(M).$
\end{proof}

In Theorem \ref{ThmHighestOrderCharDescription} we defined $\chi_N$ for each $N\in T_n$ at level $n$ of $T_n$ and thereby described the value of the characters of highest order on $T_n$.
We now recursively define $\chi_A$ for every $A \in T_n$ as follows.
 If $N$ is a node in $T_n$ with corresponding character $\chi_N$ and if $A$ is a parent of $N$ then 
define $\chi_A$ by 
\[
\chi_A(M) = \begin{cases}
|M| & \up{if $|\chi_N(M)|=|M|$};\\
-|M| & \up{if $M$ is the child of $X$ for some $X\in T_n$ with $\chi_N(X)=-|M|$};\\
0 & \up{otherwise}.
\end{cases}
\]

\begin{thm}
\label{ThmLowerOrderCharDescription}
If $\chi\in G_n^*$ is a character {\em not} of highest order then $\chi$ is a square of a higher order character. 
If $\chi$ is the square of the character $\psi$ then on a node $M\in T_n$, $\chi(M)=|M|$ if and only if $\psi(M)=\pm |M|$.
Furthermore,
$\chi(M)=-|M|$ if and only if $M$ is a child of a node $X$ on which $\psi(X)=-|X|$
On all other nodes $M$, $\chi(M)=0.$
\end{thm}

\begin{proof}
In the group $G_n = C_n\times C_n$ elements of order $2^{n-j}$ are of the form $g^{2^j}$ where $g$ is an element of highest order.
Since the dual group $G^*$ is isomorphic to $G$, then the elements of $G^*$ also have this property.

If $\chi=\psi^2$ then any element $g$ in the kernel of $\chi$ is either in the kernel of $\psi$ or is mapped by $\psi$ to $-1$. 
Thus if $g\in \ker(\chi)$ and $g\in A$ then either $\psi(A)=|M|$ or $\psi(A)=-|M|.$
If the character $\psi$ maps elements of $M$ to $-1$ then it maps elements of any child of $M$ to $\pm i$, a primitive fourth root of unity. 
This happens if and only if $\chi$ maps elements of the child of $M$ to $-1$ and thus maps $M$ to $-|M|$.

If $A$ is a node mapped to zero by $\psi$ and is not a child of a node mapped to $-|A|$ then the elements of $A$ are mapped to roots of unity who orders are at least 8. 
Therefore $\chi(A)=0.$
\end{proof}


Theorems \ref{ThmHighestOrderCharDescription} and \ref{ThmLowerOrderCharDescription} together show that for any character $\chi\in G_n^*$, there exists a unique $A\in T_n$ for which $\chi$ agrees with $\chi_A$ on $T_n$.

%
%

Let $T_n^* = \{\chi_N:N\in T_n\}$. We shall call the elements of $T_n^*$ the {\em characters} of $T_n$. We note that if $R$ is the root of $T_n$, then the unique character in $T_n^*$ that agrees with the trivial character $\chi_0 \in G_n^*$ is $\chi_R$.


\section{Canonical colorings and the PDS we enumerate}
\label{SecCanonicalColorings}

Let $n\geq2$ and $G_n = C_{2^n} \times C_{2^n}$, with orbit tree $T_n$. 

\subsection{Automorphisms}

\begin{lem}
\label{Automorphisms}
The orbit tree $T_n$ has an automorphism group of size $3 \cdot 2^{3(2^{n-1}-1)+1}$, while the automorphism group of $G_n$ has size $3\cdot 2^{4n-3}$. 
 Furthermore, the $3\cdot 2^{4n-3}$ automorphisms of $G_n$ induce a subgroup of order $3\cdot 2^{3n-2}$ within the automorphism group of $T_n$.
 \end{lem}

\begin{proof}
Given an orbit tree $T_n$, there are $3!$ ways to permute the nodes at level 1.
At each succeeding level one may always interchange the two children of a node and as there are $3(2^{n-1}-1)$ such child pairs, the total number of choices is $3! \cdot 2^{3(2^{n-1}-1)}= 3 \cdot 2^{3(2^{n-1}-1)+1}.$

That the order of the automorphism group of $G_n$ is $3\cdot 2^{4n-3}$   is given in \cite{AutomAbGps}. We give a quick proof: 
Automorphisms of $G_n$ may be identified by their actions on a pair of generators.
There are $3\cdot 2^{2n-2}$ elements of $G_n$ of order $2^n.$
Given one such element, $x$, there are $2^{2n-1}$ elements of order $2^n$ generating a subgroup that meets $\langle x \rangle$ only in the identity.
Thus there are $(3\cdot 2^{2n-2})(2^{2n-1})=3\cdot 2^{4n-3}$ automorphisms of $G_n$.

Finally, each group automorphism induces a tree automorphism.  However, the multiplier group, of order $2^{n-1}$, viewed as a subgroup of the group automorphisms, fixes every tree and so the full group of group automorphisms acts on the tree as a group of order $3\cdot 2^{4n-3}/2^{n-1}=3\cdot 2^{3n-2}.$

\end{proof}

The group of all tree automorphisms has order $3 \cdot 2^{3(2^{n-1}-1)+1}$.
The subgroup of tree automorphisms induced by group automorphisms has order $3\cdot 2^{3n-2}$ and so has  index
$2^{3(2^{n-1}-n)}$ in the tree automorphism group.
As mentioned in Section \ref{SecIntroSummary}, the traditional method of using group automorphisms to define PDS equivalence is inadequate in view of the dramatic differences between the sizes of these two automorphism groups.

\subsection{Canonical colorings} A {\em coloring} of a set $X$ is simply a function with domain $X$. A 0-1 coloring is a coloring with codomain $\{0,1\}$. We identify subsets $S\subseteq T_n$ with 0-1 colorings of $T_n$ in the obvious way: Given $S\subseteq T_n$, for $N\in T_n$, $N\in S$ if and only if $N$ is colored 1. Every PDS in $G_n$ can thus be thought of as a 0-1 coloring of $T_n$.

For a coloring $C$ of $T_n$ and a permutation $\gamma$ of $T_n$, let $\gamma(C)$ be the coloring of $T_n$ defined by $\gamma(C)(N) = C(\gamma^{-1}(N))$ for all $N\in T_n$. 

\begin{defn}
The {\em tree automorphism class} of a coloring $C$ of $T_n$ is $$\{\gamma(C):\gamma \up{ is a tree automorphism of the rooted tree }T_n\}.$$ 
\end{defn}

We now describe a distinguished representative---the {\em canonical} representative---of each such automorphism class. 
Intuitively, a 0-1 coloring $C$ of $T_n$ will be canonical if, among the colorings in its tree automorphism class, the nodes $T_n(1,\ell)$ which are colored 1 have $\ell$ minimized, and given that, the nodes $T_n(2,\ell)$ which are colored 1 have $\ell$ minimized, and so on.

To specify precisely what we mean by canonical, first for $N\in T_n$ let $(L(N),\sqsubset)$ be the linear order whose elements are the elements of the subtree of $T_n$ consisting of $N$ and its descendants, where $\sqsubset$ is the lexicographic ordering on the elements of $L(N)$ according to their ordered-pair indices (so for $T_n(i,j),T_n(i',j')\in L(N)$, we have $T_n(i,j) \sqsubset T_n(i',j')$ whenever $i<i'$, or $i=i'$ and $j<j'$). Let $L(N)_\ell$ denote the $\ell$th element (according to $\sqsubset$) of $L(N)$. If $C$ is a coloring of $T_n$, for $N,\bar N\in T_n$ and $j\in \N$ we write $C(L(N)) \equiv_j C(L(\bar N))$ if $C(L(N)_\ell) = C(L(\bar N)_\ell)$ for all $\ell\in \{1 ,\ldots, j\}$.

{\em Siblings} are distinct nodes with the same parent. For $N=T_n(i,j)$, a sibling $T_n(i,\ell)$ of $N$ is {\em younger} if $\ell<j$. For $N\in T_n$, let $P_-(N)$ be the set of nodes of $T_n$ on the path from $N$ to the root, including $N$ but excluding the root. For a node $N=T_n(n,\ell)$ at level $n$ and $M\in P_-(N)$, say $N = L(M)_j$, we say $N$ is {\em $M$-canonical in $C$} if for each younger sibling $\bar M$ of $M$ the following holds:
\[
\begin{cases}
C(N) \leq C(\bar M) & \up{if }M=N \up{ (i.e., if $j=1$)};\\
\up{If } C(L(M))\equiv_{j-1}C(L(\bar M)) \up{ then } C(N) \leq C(L(\bar M)_j) & \up{otherwise}.
\end{cases}
\]


We specify that both 0-1 colorings of $T_0$ are canonical, and the six canonical colorings of $T_1$ are the 0-1 colorings for which $T_1(1,j)=1 \implies T_1(1,\ell)=1$ for all $\ell<j$. Given a coloring $C$ of $T_n$, $\pi(C)$ denotes the restriction of $C$ to $T_{n-1}$.

\begin{defn}
\label{defCanonical}
Let $n\geq 2$. A coloring $C$ of $T_n$ is {\em canonical} if $C$ is a 0-1 coloring, $\pi(C)$ is a canonical coloring of $T_{n-1}$, and for each node $N=T_n(n,j)$ at level $n$ we have that $N$ is $M$-canonical in $C$ for all $M\in P_-(N)$.
\end{defn}

\subsection{The PDS we enumerate}To obtain a complete description of the PDS in $G_n$, we argue that it suffices to enumerate only the canonical PDS of even size. (By the {\em size} of $S\subseteq T_n$, we mean the size of $S$ thought of as a subset of $G_n$.) 

To explain why, first, if $S\subseteq T_n$ and $\gamma$ is a tree automorphism of $T_n$, then the multisets of eigenvalues of $S$ and of $\gamma(S)$ are identical. Hence by Theorem \ref{ThmPDSEig} we only need to enumerate one representative of each tree automorphism class of each PDS $S\subseteq T_n$. We choose to enumerate only canonical PDS. Second, as mentioned earlier, by complementation we only enumerate the PDS of a single parity, as $S$ is a PDS of even size if and only if $S^c$ is a PDS of odd size. Hence we choose to enumerate only the canonical PDS of even size.

We note that by Theorem \ref{ThmPDSEig}, if $S\subseteq T_n$ is a PDS of the type we aim to enumerate, then every eigenvalue of $S$ is even.

\section{LM partial colorings}
\label{SecLM}

Let $n\geq 2$ and $G_n = C_{2^n} \times C_{2^n}$, with orbit tree $T_n$. We continue to identify subsets of $T_n$ with 0-1 colorings of $T_n$. In this section we recall a result from \cite{LeifmanMuzychuk95}, which will give us a place to start our search for the canonical PDS of even size in $G_n$. Throughout this section, $X$ will denote a tuple of integers $X = (x_0, x_1 ,\ldots, x_{n-1})$, with $x_0\in \{0,1,2,3\}$ and $x_i\in \{0,1\}$ for $i>0$.

\begin{defn}
The {\em canonical homogenous set} $S_X\subseteq T_n$ is the 0-1 coloring defined by the following construction. The youngest $x_0$ children of the root are colored 1. Then, for $i$ from 1 to $n-1$, for each node $N$ of row $i$, the youngest $S_X(N) + x_i$ children of $N$ are colored 1. (Note $S_X(N)\in \{0,1\}$.) All other nodes, including the root, are colored 0.
\end{defn}

Given $X$, the {\em anti-canonical homogenous set $S'_X \subseteq T_n$} is defined in the same way as $S_X$, but replacing ``youngest" with ``oldest" in the definition. We emphasize that a canonical homogenous set is indeed a canonical coloring, and that a homogenous set is typically {\em not} a PDS. 

For $S\subseteq T_n$, let $$S|_i = \{N\in S: N\up{ is in row $i$ of $T_n$} \}.$$ A {\em block} in $T_n$ is a subset of $T_n$ of the form $\children(N)$ for some $N\in T_n$. 

We define the {\em LM partial colorings} of $T_n$ as follows. (LM stands for Leifman-Muzychuk.)

\begin{defn}
\label{DefLM}
The LM partial colorings of $T_n$ fall into three (overlapping) classes:
\begin{enumerate}
	\item \label{LMtype1} Given $X=(x_0, x_1 ,\ldots, x_{n-1})$ with $x_0\in \{0,2\}$, the {\em positive  LM partial coloring} for $X$, denoted $\LM^+_X$, is given by
	\[
	\LM^+_X = S_X|_1 \cup S_X|_n.
	\]
	\item \label{LMtype2} Given $X=(x_0, x_1 ,\ldots, x_{n-1})$ with $x_0\in \{0,2\}$, the {\em negative  LM partial coloring} for $X$, denoted $\LM^-_X$, is given by
	\[
	\LM^-_X = S_X|_1 \cup S_X|n \cup \left\{D: D \up{ is a block in row $n$ of $T_n$ and $D\cap S_X = \varnothing$} \right\}.
	\]
	\item \label{LMtype3} Given $X=(x_0, x_1 ,\ldots, x_{n-1})$ with $x_0\in \{1,3\}$, the {\em negative  LM partial coloring} for $X$, denoted $\LM^-_X$, is given by
	\[
	\LM^-_X = \left(S'_X\right)^c|_1 \cup \left(S'_X \cup \left\{D: D \up{ is a block in row $n$ of $T_n$ and $D\cap S'_X = \varnothing$} \right\}\right)^c|_n.
	\]
\end{enumerate}
\end{defn}

For example, the canonical homogenous set $S_{(2,1,0)}$ in $C_{2^3} \times C_{2^3}$ is given in Figure \ref{canonical1}.
The partial coloring $LM^{+}_{(2,1,0)}$ is obtained by zeroing out the six entries in row 2.
The partial coloring $LM^{-}_{(2,1,0)}$ is the same, except the bottom right pair of zeros is replaced with ones.

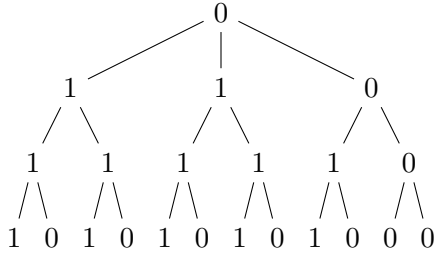
\begin{figure}
\caption{The canonical homogenous set $S_{(2,1,0)}$}%
\label{canonical1}%
\begin{tikzpicture}
	\node (00) at (0,0) {$0$};
	\node (10) at (-2,-1) {$1$};
	\node (11) at (0,-1) {$1$};
	\node (12) at (2,-1) {$0$};
	\node (20) at (-2.5,-2) {$1$};
	\node (21) at (-1.5,-2) {$1$};
	\node (22) at (-0.5,-2) {$1$};
	\node (23) at (0.5,-2) {$1$};
	\node (24) at (1.5,-2) {$1$};
	\node (25) at (2.5,-2) {$0$};
	\node (30) at (-2.75,-3) {1};
	\node (31) at (-2.25,-3) {0};
	\node (32) at (-1.75,-3) {1};
	\node (33) at (-1.25,-3) {0};
	\node (34) at (-0.75,-3) {1};
	\node (35) at (-0.25,-3) {0};
	\node (36) at (0.25,-3) {1};
	\node (37) at (0.75,-3) {0};
	\node (38) at (1.25,-3) {1};
	\node (39) at (1.75,-3) {0};
	\node (310) at (2.25,-3) {0};
	\node (311) at (2.75,-3) {0};
	\draw (00)--(10);
	\draw (00)--(11);
	\draw (00)--(12);
	\draw (10)--(20);
	\draw (10)--(21);
	\draw (11)--(22);
	\draw (11)--(23);
	\draw (12)--(24);
	\draw (12)--(25);
	\draw (20)--(30);
	\draw (20)--(31);
	\draw (21)--(32);
	\draw (21)--(33);
	\draw (22)--(34);
	\draw (22)--(35);
	\draw (23)--(36);
	\draw (23)--(37);
	\draw (24)--(38);
	\draw (24)--(39);
	\draw (25)--(310);
	\draw (25)--(311);
\end{tikzpicture}
\end{figure}

Every LM partial coloring of $T_n$ is actually of type (\ref{LMtype2}) or (\ref{LMtype3}), and those of type (\ref{LMtype2}) and (\ref{LMtype3}) overlap in exactly one coloring; hence there are $2^{n+1}-1$ distinct LM partial colorings of $T_n$. We have mentioned those of type (\ref{LMtype1}) only to state Theorem \ref{ThmLM} below precisely.

Using the observations that if $S\subseteq T_n$ is a canonical PDS then $S^h$ is a canonical coloring (see \cite{LeifmanMuzychuk95} for the meaning of the notation $S^h$), and that the complement of a canonical homogenous set is an anti-canonical homogenous set and vice versa, we translate Propositions 6.3 and 6.4 of \cite{LeifmanMuzychuk95} as follows:

\begin{thm}
\label{ThmLM}
Let $S\subseteq T_n$ be a canonical PDS of even size. Then $S$ is either a positive or negative Latin square PDS. Furthermore: 
\begin{enumerate}
	\item If $S$ is a positive Latin square PDS, then $S|_0 \cup S|_1\cup S|_n = \LM^+_X$ for some tuple $X=(x_0, x_1 ,\ldots, x_{n-1})$, with $x_0\in \{0,2\}$.
	\item If $S$ is a negative Latin square PDS, then $S|_0 \cup S|_1\cup S|_n = \LM^-_X$ for some tuple $X$.
\end{enumerate}
\end{thm}

Combined with the discussion following Definition \ref{DefLM}, Theorem \ref{ThmLM} says that there are only $2^{n+1}-1$ possible color combinations for rows $0$, $1$, and $n$ of a canonical PDS of even size in $G_n$. Nevertheless we note that extensions of LM partial colorings to canonical PDS of even size are fairly rare. For instance, for $n=5$ there are 63 unique LM partial colorings of $T_n$, and it turns out that only 7 of them extend to yield the 12 canonical PDS of even size in $G_5$.


\section{Our algorithm}
\label{SecAlgorithm}

Let $n\geq 2$ and $G_n = C_{2^n} \times C_{2^n}$, with orbit tree $T_n$. Our goal is to find all (regular nontrivial) canonical PDS $S\subseteq T_n$ of even size. By Theorem \ref{ThmPDSEig} this amounts to finding all canonical 0-1 colorings $C$ of $T_n$ with three distinct even eigenvalues $k, r, s$, such that $k, r>0, s<0, s=r-2^n, k \equiv r \equiv s \pmod{2^n}$, and $\chi_N(C)\in\{r, s\}$ for all $N\in T_n\setminus T_n(0,0)$ (and hence $r\in R = \{2, 4,\ldots, 2^{n}-2\}$). While it is fairly straightforward to describe a search tree algorithm that generates the canonical colorings of $T_n$, it would be inefficient to generate all of the canonical colorings of $T_n$ and then check which ones meet the additional conditions we are looking for. Instead we will traverse a search forest that takes all of our conditions into account simultaneously, intelligently eliminating large branches which will never lead to a PDS or which will never lead to a canonical coloring. 

\subsection{Outline of the algorithm}

The starting point for our search is the collection of LM partial colorings of $T_n$. 

\begin{defn}
Let $C$ be an LM partial coloring of $T_n$. If $S\subseteq T_n$, we say $S$ {\em extends} $C$ if $S|_0 = C|_0$,  $S|_1 = C|_1$, and $S|_n = C|_n$.
\end{defn}

Given an LM partial coloring $C$, we will begin by computing a restricted set $R'\subseteq R$ such that any PDS extending $C$ must have its positive nontrivial eigenvalue $r\in R'$. We call $R'$ the set of {\em feasible nontrivial positive eigenvalues} for $C$. For each $C$ we will compute $R'$, and then for each $r\in R'$ we will perform a tree search to output every canonical PDS $S$ extending $C$ whose positive nontrivial eigenvalue is $r$. That is, the outline of our algorithm is as follows:

\begin{lstlisting}[
  mathescape,
  columns=fullflexible,
  basicstyle=\fontfamily{lmvtt}\selectfont,
]
(1) Generate the LM partial colorings of $T_n$.
Then for each LM partial coloring $C$
{
  (2) Compute the feasible nontrivial positive eigenvalues $R'$ for $C$.
  Then for each $r\in R$
  {
    (3) Compute and output each nontrivial canonical PDS $S$ extending $C$ whose positive
        nontrivial eigenvalue is $r$.
  }
}	
\end{lstlisting}
The computations necessary to carry out step (1) were described in Section \ref{SecLM}. We explain the computations for steps (2) and (3) in Sections \ref{SecFeasibleEigs} and \ref{SecTreeSearch}, respectively.

\subsection{Feasible eigenvalues}
\label{SecFeasibleEigs}

Let $C$ be an LM partial coloring of $T_n$. If $S$ is a PDS extending $C$, then the positive nontrivial eigenvalue $r$ of $S$ lies in $R = \{2, 4,\ldots, 2^{n}-2\}.$ 
In this section we describe how to find a subset $R'$ of $R$, which we call the {\em feasible nontrivial positive eigenvalues} for $C$, which has the property that if $S$ is any PDS extending $C$, then the eigenvalue $r$ of $S$ lies in $R'$. 

We define a new coloring $C'$ as follows: For any node $N\in T_n$,
\[
C'(N) = 
\begin{cases}
C(N) &\up{if $N$ is in row $0$, $1$, or $n$ of $T_n$};\\
x_N &\up{otherwise},
\end{cases}
\]
where $x_N$ is a variable with range $\{0,1\}$. 
Hence the extensions of $C$ are exactly the sets obtainable from $C'$ by assigning values to every variable $x_N$, as $N$ ranges over the nodes in rows $2 ,\ldots, n-1$ of $T_n$.

Let $r\in R$. 
We shall use what is essentially a Gaussian elimination method to decide whether or not $r\in R'$, with the twist that some of the entries of the augmented matrix we are about to describe are sets instead of scalars. Let $\chi\in G_n^*$ be a nontrivial character.

\begin{prop}Suppose $\chi$ is of highest order; say $\chi = \chi_N$ for some node $N$ at row $n$ of $T_n$. Let $\bar{N}$ be the sibling of $N$, and suppose further that $C(N) \neq C(\bar{N})$. If there is an assignment of the variables $x_N$ in $C'$ that produces a PDS $S$ with positive nontrivial eigenvalue $r$, then:
\begin{itemize}
	\item If $C(N)=1$, then $\chi_N(S) = r$.
	\item If $C(N)=0$, then $\chi_N(S) = r-2^n$.
\end{itemize}
\end{prop}

\begin{proof}Suppose there is an assignment of the variables $x_N$ in $C'$ that produces a PDS $S$ with positive nontrivial eigenvalue $r$. Then $\chi(S)\in \{r,r-2^n\}$ for any nontrivial $\chi\in G_n^*$. 

Suppose $C(N) = 1$. Then $C({\bar{N}}) = 0$. Since $\chi_N$ and $\chi_{\bar{N}}$ agree everywhere except on $N$ and $\bar{N}$, by the definition of $\chi_N$ we have $\chi_N(S) - \chi_{\bar{N}}(S) = 2^n$. Thus $\chi_N(S) > \chi_{\bar{N}}(S)$, so $\chi_N(S) = r$.

The case where $C(N) = 0$ is similar. 
\end{proof}


We now define a statement $E_\chi$ that must be satisfied for there to be a PDS extending $C$ with positive nontrivial eigenvalue $r$. If $\chi$ is a character of highest order, with $N$ the node at row $n$ of $T_n$ such that $\chi = \chi_N$ and $\bar{N}$ the sibling of $N$, and $C(N)\neq C(\bar{N})$, then in the case that $C(N)=1$ we define the statement $E_{\chi_N}$ given by
\[
E_{\chi_N} = ``\chi_N(C') \in \{r\},"
\]
while in the case that $C(N)=0$ we define the statement
\[
E_{\chi_N} = ``\chi_N(C') \in \{r-2^n\}."
\]
On the other hand, if $C(N) = C(\bar{N})$ or if $\chi$ is a nontrivial character of lower order, we define the statement
\[
E_{\chi} = ``\chi(C') \in \{r, r-2^n\}."
\]
Note that for any $\chi$, $\chi(C')$ is a linear expression with variables $\{x_N\}$. We therefore have a system of ``linear statements" $\{E_\chi\}$, as $\chi$ ranges over the nontrivial characters of $G_n$. The system of statements $\{E_\chi\}$ is {\em consistent} if there exists an assignment to the variables $\{x_N\}$ that makes every statement in the system true. By Theorem \ref{ThmPDSEig}, there exists a PDS extending $C$ having positive nontrivial eigenvalue $r$ only if the system $\{E_\chi\}$ is consistent.

We now build an augmented matrix $[A|v]$ from the statements $E_\chi$ in the same way one usually builds the augmented matrix of a system of linear equations. We index the columns of $A$ by the variables $x_N$ and the rows of $A$ and $v$ by the nontrivial characters $\chi$. The row of $A$ for $\chi$ has the coefficient of $x_N$ in $\chi(C')$ in the $x_N$ column, and the entry in the $\chi$ row of $v$ is the set on the right hand side of the statement $E_\chi$. While the normal row operations used in Gaussian elimination can be performed on the matrix $[A|v]$ (with the usual operations for multiplication of a set by a scalar: $cM=\{cm:m\in M\}$ and for the addition of sets: $M+N=\{m+n:m\in M, n\in N\}$), the row operation that replaces one row with the sum of that row and a scalar multiple of another row is {\em not} reversible, and hence Gaussian elimination on $[A|v]$ alone cannot solve or even prove consistency of our associated system of statements. However, it can reveal inconsistency---if $[A|v]$ represents a consistent system, then so does the matrix $[\rref(A)|w]$ obtained by Gaussian elimination---and this is what we use to determine whether or not $r$ is feasible. If there is a row of $\rref(A)$ which is all zeroes and the corresponding entry of $w$ does not contain a $0$, then $[\rref(A)|w]$ does not represent a consistent system, and therefore neither does $[A|v]$, and hence we deem $r$ to be {\em non-feasible}, and $r\notin R'$. Otherwise $r$ is {\em feasible}, and $r\in R'$.

We note that typically there will be several feasible values of $r$ for which there are actually no canonical PDS extensions of $C$ having positive nontrivial eigenvalue $r$. However $|R'|$ will be much smaller than $|R|$, and the next step of our algorithm takes $C$ and a positive even integer $r$ and performs a relatively expensive tree search to find all canonical PDS extending $C$ with positive nontrivial eigenvalue $r$. In searching for all canonical PDS extending $C$, for larger values of $n$ it is faster to eliminate the non-feasible $r$ values by performing the computation in this section rather than eliminating them by subjecting them to the tree search algorithm in the next section.

\subsection{A tree search driven by modular equations}
\label{SecTreeSearch}

Let $C$ be an LM partial coloring and $r\in\{2, 4, \ldots, 2^{n}-2\}$. In this section we describe a tree search method that finds all canonical PDS extending $C$ having positive nontrivial eigenvalue $r$. We begin by forming $C'$ as we did in Section \ref{SecFeasibleEigs} by setting, for $N\in T_n$,
\[
C'(N) = 
\begin{cases}
C(N) &\up{if $N$ is in row $0$, $1$, or $n$ of $T_n$};\\
x_N &\up{otherwise},
\end{cases}
\]
where $x_N$ is a variable having range $\{0,1\}$. Hence any canonical PDS extending $C$ corresponds to some assignment of values to the variables $x_N$ in the definition of $C'$. $C'$ is the root of our search tree. In general, a node $Q$ of our search tree will be a modification of its parent $P$, with the properties that $\forall M\in T_n$, $Q(M)$ is either $0$, $1$, or a linear expression containing one of the variables $x_N$, and that if $P(M)$ has been assigned a value of either 0 or 1 then $Q(M)=P(M)$. We will say that $S$ is a PDS {\em extending} $Q$ if $S$ is a PDS and may be obtained by assigning values to the variables $x_N$ in $Q$. We also say $Q$ is {\em definitively non-canonical} if there exists any $N\in T_n$ at any level $i$ for which $N$ is not $M$-canonical in $\pi^{n-i}(Q)$ for some $M\in P_{-}(N)$ (where $\pi^{0}$ is the identity map and $\pi^{n-i}$ is the $n-i$ fold composition of $\pi$ with itself). Definitively non-canonical nodes $Q$ cannot be extended to canonical colorings and so will have no children.

To form the children of a node $Q$ of our search tree we use the following four steps:
\begin{enumerate}
	\item If there is a PDS $S$ extending $Q$ having positive nontrivial eigenvalue $r$ then 
\[
\chi(Q) \equiv \chi(S) \equiv r\pmod{2^n}
\]
and hence also
\begin{align*}
\chi(Q) &\equiv r\pmod{2^{n-1}}\\
\chi(Q) &\equiv r\pmod{2^{n-2}}\\
&\vdots\\
\chi(Q) &\equiv r\pmod{4}
\end{align*}
for any $\chi\in T_n^*$. If any expression $\chi(Q)$ contains exactly one variable $x_N$ after being reduced mod ${2^k}$ for any $k$ we test whether or not the statement $\chi(Q) \equiv r \pmod{2^k}$ is true after making the substitutions $x_N=0$ and $x_N=1$. If exactly one of these substitutions $x_N=i$ gives a true statement then we set the value of the variable $x_N$ to $i$. If neither statement is true then $Q$ cannot be extended to a PDS with nontrivial eigenvalue $r$, so $Q$ has no children.

Similarly, if any expression $\chi(Q)$ contains exactly two variables $x_M, x_N$ after being reduced mod ${2^k}$, we test whether or not the statement $\chi(Q) \equiv r \pmod{2^k}$ is true after making all four substitutions $x_M=i, x_N=j$ for $i,j\in \{0,1\}$. 
If we find that the statement is true exactly in the two cases that $x_M=x_N$ we set $x_N = x_M$ for $Q$ and hence also all of $Q$'s descendents. If we find that the statement is true exactly in the two cases that $x_M\neq x_N$ we set $x_N\neq x_M$ (that is, $x_N = 1-x_M$) for $Q$ and hence also all of $Q$'s descendents, and further we check both assignments $x_M=0$ and $x_M=1$ to see if either would cause $Q$ to be definitively non-canonical: If exactly one makes $Q$ definitively non-canonical we update $x_M$ and $x_N$ to $1$ or $0$ accordingly so that $Q$ is not definitively non-canonical, and if both make $Q$ definitively non-canonical then $Q$ has no children. We note that testing the $Q(L(M))\equiv_{j-1}Q(L(\bar M))$ condition for a node to be $M$-canonical in $Q$ can be done even if the values in $Q(L(M))$ and $Q(L(\bar M))$ are variable expressions.
Finally if the statement is true in none of the four cases then $Q$ cannot be extended to a PDS, so $Q$ has no children. 

Out of any clique of variables which are equal or unequal in pairs where no variable has yet been assigned a value, there is only one independent variable, which we take to be the earliest variable from the clique in the linear order $L(T_n(0,0))$.

This step is looped until no further changes to $Q$ are possible. We note that this step typically reduces the number of independent variables drastically when $Q=C'$, and continues to be useful for reducing the number of independent variables for deeper nodes $Q$ in our search tree.

	\item We apply characters to the updated version of $Q$ from the previous step. If for any nontrivial character $\chi$ we find that $\chi(Q)$ is a constant and $\chi(Q)\notin\{r,r-2^n\}$ then $Q$ cannot be extended to a PDS, so $Q$ has no children. We note that $\chi(Q)$ can be a constant even in the case that there are nodes $N\in T_n$ in the support of $\chi$ for which $Q(N)$ is a variable-containing expression, as variable cancellation can occur in the evaluation of $\chi(Q)$. (The {\em support} of $\chi$ is $\{N\in T_n:\chi(N)\neq 0\}$.)

	\item If $Q$ is definitively non-canonical $Q$ will not have any children.
	
	\item If at this point $Q$ has not been designated as having no children and $Q$ has at least one independent variable in its range, $Q$ will have the two children $Q_0$ and $Q_1$, where $Q_i$ is obtained by taking the first independent variable in the range of $Q$ and assigning it the value $i$.
\end{enumerate}

The output of our algorithm are the leaves of our search tree (that is, the colorings $Q$ in our search tree with no unassigned variable in the range of $Q$) for which $Q$ is canonical, $\chi(Q)\in \{r,r-2^n\}$ for all nontrivial $\chi\in T_n^*$, and also have the property that $Q$ has exactly three eigenvalues.
Since our algorithm only cuts branches in our search tree that cannot be extended to canonical colorings having nontrivial eigenvalues in $\{r,r-2^n\}$, the output of our algorithm is all the canonical PDS extending $C$ having positive nontrivial eigenvalue $r$.

\section{Examples}
\label{SecExamples}

\subsection{An example in $C_8 \times C_8$}
\label{SecC8ex}

Let $C'$ be the coloring of $T_3$ (the orbit tree of $C_{2^3}\times C_{2^3}$) depicted in Figure \ref{FigC8st}. We have renamed the variables $x_{T_3(2,i)}$ as $a_i$ for ease of writing, and for instance $C'(T_3(2,3)) = x_{T_3(2,3)} = a_3$, $C'(T_3(3,0)) = 1$, and $C'(T_3(3,1)) = 0$. We note that $C'$ was obtained from $\LM^-_{(2,1,0)}$. Let us take $r=4$. We now apply the algorithm from Section \ref{SecTreeSearch} to this combination of $C'$ and $r$. (It will turn out we will only need to apply step (1) of the algorithm to fill in the values of all six variables $a_i$.)

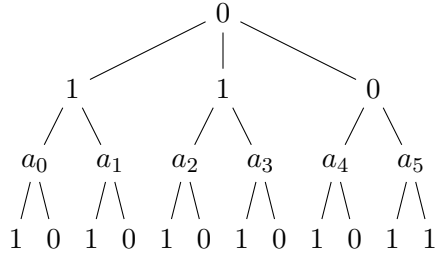
\begin{figure}%
\caption{$C_8\times C_8$ example, $C'$}%
\label{FigC8st}%
\begin{tikzpicture}
	\node (00) at (0,0) {$0$};
	\node (10) at (-2,-1) {$1$};
	\node (11) at (0,-1) {$1$};
	\node (12) at (2,-1) {$0$};
	\node (20) at (-2.5,-2) {$a_0$};
	\node (21) at (-1.5,-2) {$a_1$};
	\node (22) at (-0.5,-2) {$a_2$};
	\node (23) at (0.5,-2) {$a_3$};
	\node (24) at (1.5,-2) {$a_4$};
	\node (25) at (2.5,-2) {$a_5$};
	\node (30) at (-2.75,-3) {1};
	\node (31) at (-2.25,-3) {0};
	\node (32) at (-1.75,-3) {1};
	\node (33) at (-1.25,-3) {0};
	\node (34) at (-0.75,-3) {1};
	\node (35) at (-0.25,-3) {0};
	\node (36) at (0.25,-3) {1};
	\node (37) at (0.75,-3) {0};
	\node (38) at (1.25,-3) {1};
	\node (39) at (1.75,-3) {0};
	\node (310) at (2.25,-3) {1};
	\node (311) at (2.75,-3) {1};
	\draw (00)--(10);
	\draw (00)--(11);
	\draw (00)--(12);
	\draw (10)--(20);
	\draw (10)--(21);
	\draw (11)--(22);
	\draw (11)--(23);
	\draw (12)--(24);
	\draw (12)--(25);
	\draw (20)--(30);
	\draw (20)--(31);
	\draw (21)--(32);
	\draw (21)--(33);
	\draw (22)--(34);
	\draw (22)--(35);
	\draw (23)--(36);
	\draw (23)--(37);
	\draw (24)--(38);
	\draw (24)--(39);
	\draw (25)--(310);
	\draw (25)--(311);
\end{tikzpicture}
\end{figure}

Let us write $\chi_{(i,j)}$ for $\chi_{T_3(i,j)}$. 
First, using $\chi_{(3,0)}$ we see that $\chi_{(3,0)}(C') = 4 + 2a_0-2a_1+1-1$ and for there to be a PDS extending $C'$ we must have $\chi_{(3,0)}(C') \equiv r\pmod{4}$; that is,
\[
2a_0-2a_1 \equiv 0 \pmod{4}.
\]
Testing all four combinations of possibilities for $a_0$ and $a_1$ reveals that we must have $a_0=a_1$. Similarly, applying $\chi_{(3,4)}$ to $C'$ and reducing mod 4 reveals that $a_2=a_3$.

Next, $\chi_{(3,8)}(C') = 4 +2a_4 -2a_5-2$, and we require $\chi_{(3,8)}(C') \equiv r\pmod{4}$; that is,
\[
2+2a_4-2a_5 \equiv 0 \pmod{4}.
\]
Testing all four combinations of possibilities for $a_4$ and $a_5$ reveals that we must have $a_4\neq a_5$, i.e., $a_5 = 1-a_4$. Since we are searching only for canonical colorings we must therefore have $a_4=1$ and $a_5=0$. 
At this point the root $C'$ of our search tree has morphed into the coloring $Q$ depicted in Figure \ref{FigC8_2}.

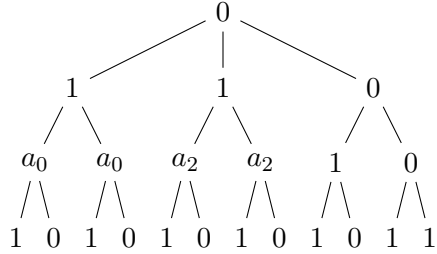
\begin{figure}%
\caption{$C_8\times C_8$ example, $Q$}%
\label{FigC8_2}%
\begin{tikzpicture}
	\node (00) at (0,0) {$0$};
	\node (10) at (-2,-1) {$1$};
	\node (11) at (0,-1) {$1$};
	\node (12) at (2,-1) {$0$};
	\node (20) at (-2.5,-2) {$a_0$};
	\node (21) at (-1.5,-2) {$a_0$};
	\node (22) at (-0.5,-2) {$a_2$};
	\node (23) at (0.5,-2) {$a_2$};
	\node (24) at (1.5,-2) {$1$};
	\node (25) at (2.5,-2) {$0$};
	\node (30) at (-2.75,-3) {1};
	\node (31) at (-2.25,-3) {0};
	\node (32) at (-1.75,-3) {1};
	\node (33) at (-1.25,-3) {0};
	\node (34) at (-0.75,-3) {1};
	\node (35) at (-0.25,-3) {0};
	\node (36) at (0.25,-3) {1};
	\node (37) at (0.75,-3) {0};
	\node (38) at (1.25,-3) {1};
	\node (39) at (1.75,-3) {0};
	\node (310) at (2.25,-3) {1};
	\node (311) at (2.75,-3) {1};
	\draw (00)--(10);
	\draw (00)--(11);
	\draw (00)--(12);
	\draw (10)--(20);
	\draw (10)--(21);
	\draw (11)--(22);
	\draw (11)--(23);
	\draw (12)--(24);
	\draw (12)--(25);
	\draw (20)--(30);
	\draw (20)--(31);
	\draw (21)--(32);
	\draw (21)--(33);
	\draw (22)--(34);
	\draw (22)--(35);
	\draw (23)--(36);
	\draw (23)--(37);
	\draw (24)--(38);
	\draw (24)--(39);
	\draw (25)--(310);
	\draw (25)--(311);
\end{tikzpicture}
\end{figure}

Continuing with characters of higher order, next we examine $\chi_{(2,0)}(Q) = 4-4+2a_0+2a_0 - 2a_2 - 2a_2 +1-1$. Reducing mod $8$ we obtain
\[
4a_0 - 4a_2 \equiv 4 \pmod{8},
\]
which implies $a_0\neq a_2$. Finally, since we seek only canonical colorings we must have $a_0=1, a_1=0$. $Q$ has become the coloring in Figure \ref{FigC8_3}.

\begin{figure}%
\caption{$C_8\times C_8$ example, final result}%
\label{FigC8_3}%
\begin{tikzpicture}
	\node (00) at (0,0) {$0$};
	\node (10) at (-2,-1) {$1$};
	\node (11) at (0,-1) {$1$};
	\node (12) at (2,-1) {$0$};
	\node (20) at (-2.5,-2) {$1$};
	\node (21) at (-1.5,-2) {$1$};
	\node (22) at (-0.5,-2) {$0$};
	\node (23) at (0.5,-2) {$0$};
	\node (24) at (1.5,-2) {$1$};
	\node (25) at (2.5,-2) {$0$};
	\node (30) at (-2.75,-3) {1};
	\node (31) at (-2.25,-3) {0};
	\node (32) at (-1.75,-3) {1};
	\node (33) at (-1.25,-3) {0};
	\node (34) at (-0.75,-3) {1};
	\node (35) at (-0.25,-3) {0};
	\node (36) at (0.25,-3) {1};
	\node (37) at (0.75,-3) {0};
	\node (38) at (1.25,-3) {1};
	\node (39) at (1.75,-3) {0};
	\node (310) at (2.25,-3) {1};
	\node (311) at (2.75,-3) {1};
	\draw (00)--(10);
	\draw (00)--(11);
	\draw (00)--(12);
	\draw (10)--(20);
	\draw (10)--(21);
	\draw (11)--(22);
	\draw (11)--(23);
	\draw (12)--(24);
	\draw (12)--(25);
	\draw (20)--(30);
	\draw (20)--(31);
	\draw (21)--(32);
	\draw (21)--(33);
	\draw (22)--(34);
	\draw (22)--(35);
	\draw (23)--(36);
	\draw (23)--(37);
	\draw (24)--(38);
	\draw (24)--(39);
	\draw (25)--(310);
	\draw (25)--(311);
\end{tikzpicture}
\end{figure}
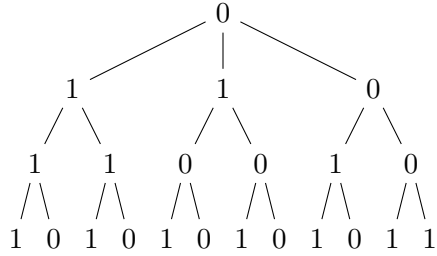

We check all eigenvalues of our final coloring and find that all nontrivial eigenvalues are $r=4$ or $r-8 = -4$, and that the trivial eigenvalue is 36; hence the resulting coloring of $T_3$ is one of the PDS in $G_3$ we seek. (It is the sixth PDS in Table \ref{PDSn3} below.)

\subsection{A non-example in $C_{8}\times C_{8}$} Let $C'$ be the coloring of $T_3$ depicted in Figure \ref{FigC8_non_st}. We note that $C'$ was obtained from $\LM_{(0,1,0)}^-$. Again we have renamed the variables $x_{T_3(2,i)}$ as $a_i$, and the characters $\chi_{T_3(i,j)}$ as $\chi_{(i,j)}$. We take $r=2$ or $r=6$, and apply the algorithm from Section \ref{SecTreeSearch} to this combination of $C'$ and $r$.

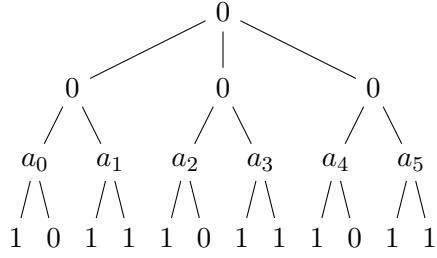
\begin{figure}%
\caption{$C_8\times C_8$ non-example, $C'$}%
\label{FigC8_non_st}%
\begin{tikzpicture}
	\node (00) at (0,0) {$0$};
	\node (10) at (-2,-1) {$0$};
	\node (11) at (0,-1) {$0$};
	\node (12) at (2,-1) {$0$};
	\node (20) at (-2.5,-2) {$a_0$};
	\node (21) at (-1.5,-2) {$a_1$};
	\node (22) at (-0.5,-2) {$a_2$};
	\node (23) at (0.5,-2) {$a_3$};
	\node (24) at (1.5,-2) {$a_4$};
	\node (25) at (2.5,-2) {$a_5$};
	\node (30) at (-2.75,-3) {1};
	\node (31) at (-2.25,-3) {0};
	\node (32) at (-1.75,-3) {1};
	\node (33) at (-1.25,-3) {1};
	\node (34) at (-0.75,-3) {1};
	\node (35) at (-0.25,-3) {0};
	\node (36) at (0.25,-3) {1};
	\node (37) at (0.75,-3) {1};
	\node (38) at (1.25,-3) {1};
	\node (39) at (1.75,-3) {0};
	\node (310) at (2.25,-3) {1};
	\node (311) at (2.75,-3) {1};
	\draw (00)--(10);
	\draw (00)--(11);
	\draw (00)--(12);
	\draw (10)--(20);
	\draw (10)--(21);
	\draw (11)--(22);
	\draw (11)--(23);
	\draw (12)--(24);
	\draw (12)--(25);
	\draw (20)--(30);
	\draw (20)--(31);
	\draw (21)--(32);
	\draw (21)--(33);
	\draw (22)--(34);
	\draw (22)--(35);
	\draw (23)--(36);
	\draw (23)--(37);
	\draw (24)--(38);
	\draw (24)--(39);
	\draw (25)--(310);
	\draw (25)--(311);
\end{tikzpicture}
\end{figure}

First, $\chi_{(3,0)}(C') = 4 +2a_0-2a_1$. Reducing mod 4 and setting equivalent to $r$, we obtain $2a_0-2a_1\equiv 2 \pmod{4}$, and hence $a_0\neq a_1$. Since we seek only canonical colorings we must have $a_0=1, a_1=0$. 

Similarly, using $\chi_{(3,4)}$ and $\chi_{(3,8)}$ we obtain $a_2=1,a_3=0,a_4=1,a_5=0$.
At this point $C'$ has become the coloring $Q$ depicted in Figure \ref{FigC8_non_2}.

\begin{figure}%
\caption{$C_8\times C_8$ non-example, $Q$}%
\label{FigC8_non_2}%
\begin{tikzpicture}
	\node (00) at (0,0) {$0$};
	\node (10) at (-2,-1) {$0$};
	\node (11) at (0,-1) {$0$};
	\node (12) at (2,-1) {$0$};
	\node (20) at (-2.5,-2) {$1$};
	\node (21) at (-1.5,-2) {$0$};
	\node (22) at (-0.5,-2) {$1$};
	\node (23) at (0.5,-2) {$0$};
	\node (24) at (1.5,-2) {$1$};
	\node (25) at (2.5,-2) {$0$};
	\node (30) at (-2.75,-3) {1};
	\node (31) at (-2.25,-3) {0};
	\node (32) at (-1.75,-3) {1};
	\node (33) at (-1.25,-3) {1};
	\node (34) at (-0.75,-3) {1};
	\node (35) at (-0.25,-3) {0};
	\node (36) at (0.25,-3) {1};
	\node (37) at (0.75,-3) {1};
	\node (38) at (1.25,-3) {1};
	\node (39) at (1.75,-3) {0};
	\node (310) at (2.25,-3) {1};
	\node (311) at (2.75,-3) {1};
	\draw (00)--(10);
	\draw (00)--(11);
	\draw (00)--(12);
	\draw (10)--(20);
	\draw (10)--(21);
	\draw (11)--(22);
	\draw (11)--(23);
	\draw (12)--(24);
	\draw (12)--(25);
	\draw (20)--(30);
	\draw (20)--(31);
	\draw (21)--(32);
	\draw (21)--(33);
	\draw (22)--(34);
	\draw (22)--(35);
	\draw (23)--(36);
	\draw (23)--(37);
	\draw (24)--(38);
	\draw (24)--(39);
	\draw (25)--(310);
	\draw (25)--(311);
\end{tikzpicture}
\end{figure}
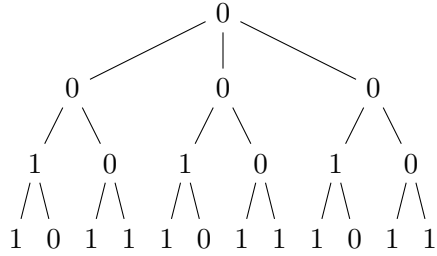

For $Q$ to be a PDS it must obey Theorem \ref{ThmPDSEig}. We apply characters and find:
\begin{itemize}
	\item $\chi_{(0,0)}(Q) = 42$.
	\item $\chi_{(3,0)}(Q) = 6$.
	\item $\chi_{(2,1)}(Q) = 2$.
\end{itemize}
Since $Q$ has too many positive eigenvalues $Q$ is not (or in more general cases, cannot be extended to) a PDS.

\subsection{An example in $C_{16}\times C_{16}$} Let $C'$ be the coloring of $T_4$ depicted in Figure \ref{FigC16st}. Again we have renamed the variables of the form $x_{T_4(i,j)}$ as variables of the form $a_k$, and we continue to write $\chi_{(i,j)}$ for $\chi_{T_4(i,j)}$. We note that $C'$ was obtained from $\LM_{(2,1,1,0)}^-$. For this example we take $r=8$ and we apply the algorithm from Section \ref{SecTreeSearch} to this combination of $C'$ and $r$. Despite the apparent similarity to the example from Section \ref{SecC8ex} several new wrinkles arise. 

\begin{figure}%
\caption{$C_{16}\times C_{16}$ example, $C'$}%
\label{FigC16st}%
\begin{tikzpicture}
	\node (00) at (0,0) {$0$};
	\node (10) at (-4,-1) {$1$};
	\node (11) at (0,-1) {$1$};
	\node (12) at (4,-1) {$0$};
	\node (20) at (-5,-2) {$a_0$};
	\node (21) at (-3,-2) {$a_1$};
	\node (22) at (-1,-2) {$a_2$};
	\node (23) at (1,-2) {$a_3$};
	\node (24) at (3,-2) {$a_4$};
	\node (25) at (5,-2) {$a_5$};
	\node (30) at (-5.5,-3) {$a_6$};
	\node (31) at (-4.5,-3) {$a_7$};
	\node (32) at (-3.5,-3) {$a_8$};
	\node (33) at (-2.5,-3) {$a_9$};
	\node (34) at (-1.5,-3) {$a_{10}$};
	\node (35) at (-0.5,-3) {$a_{11}$};
	\node (36) at (0.5,-3) {$a_{12}$};
	\node (37) at (1.5,-3) {$a_{13}$};
	\node (38) at (2.5,-3) {$a_{14}$};
	\node (39) at (3.5,-3) {$a_{15}$};
	\node (310) at (4.5,-3) {$a_{16}$};
	\node (311) at (5.5,-3) {$a_{17}$};
\node (40) at (-5.75,-4) {$1$};
\node (41) at (-5.25,-4) {$0$};
\node (42) at (-4.75,-4) {$1$};
\node (43) at (-4.25,-4) {$0$};
\node (44) at (-3.75,-4) {$1$};
\node (45) at (-3.25,-4) {$0$};
\node (46) at (-2.75,-4) {$1$};
\node (47) at (-2.25,-4) {$0$};
\node (48) at (-1.75,-4) {$1$};
\node (49) at (-1.25,-4) {$0$};
\node (410) at (-0.75,-4) {$1$};
\node (411) at (-0.25,-4) {$0$};
\node (412) at (0.25,-4) {$1$};
\node (413) at (0.75,-4) {$0$};
\node (414) at (1.25,-4) {$1$};
\node (415) at (1.75,-4) {$0$};
\node (416) at (2.25,-4) {$1$};
\node (417) at (2.75,-4) {$0$};
\node (418) at (3.25,-4) {$1$};
\node (419) at (3.75,-4) {$0$};
\node (420) at (4.25,-4) {$1$};
\node (421) at (4.75,-4) {$0$};
\node (422) at (5.25,-4) {$1$};
\node (423) at (5.75,-4) {$1$};
	\draw (00)--(10);
	\draw (00)--(11);
	\draw (00)--(12);
	\draw (10)--(20);
	\draw (10)--(21);
	\draw (11)--(22);
	\draw (11)--(23);
	\draw (12)--(24);
	\draw (12)--(25);
	\draw (20)--(30);
	\draw (20)--(31);
	\draw (21)--(32);
	\draw (21)--(33);
	\draw (22)--(34);
	\draw (22)--(35);
	\draw (23)--(36);
	\draw (23)--(37);
	\draw (24)--(38);
	\draw (24)--(39);
	\draw (25)--(310);
	\draw (25)--(311);
	\draw (30)--(40);
	\draw (30)--(41);
	\draw (31)--(42);
	\draw (31)--(43);
	\draw (32)--(44);
	\draw (32)--(45);
	\draw (33)--(46);
	\draw (33)--(47);
	\draw (34)--(48);
	\draw (34)--(49);
	\draw (35)--(410);
	\draw (35)--(411);
	\draw (36)--(412);
	\draw (36)--(413);
	\draw (37)--(414);
	\draw (37)--(415);
	\draw (38)--(416);
	\draw (38)--(417);
	\draw (39)--(418);
	\draw (39)--(419);
	\draw (310)--(420);
	\draw (310)--(421);
	\draw (311)--(422);
	\draw (311)--(423);
\end{tikzpicture}
\end{figure}
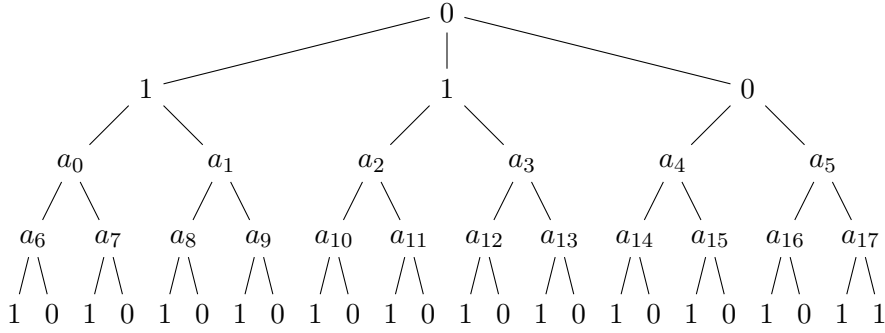

First, we have $\chi_{(4,0)}(C') = 8 + 4a_6 - 4a_7 +2a_0-2a_1+1-1$, and reducing mod 4 we have $2a_0-2a_1\equiv 0 \pmod{4}$. 
Hence $a_0=a_1$. Similarly, using $\chi_{(4,8)}$ we obtain $a_2=a_3$.  
Next, we examine $\chi_{(4,22)}(C') = 8-8+4a_{17}-4a_{16}+2a_5-2a_4-2$. Reducing mod 4 we obtain $2a_5-2a_4-2\equiv 0 \pmod{4}$. Hence $a_4\neq a_5$, and since we seek only canonical colorings we have $a_4=1, a_5=0$.

We now apply $\chi_{(4,0)}$ again. We have $\chi_{(4,0)}(C') = 8+4a_6-4a_7+2a_0-2a_1+1-1$. 
After making the substitution $a_0=a_1$ and reducing mod 8 we obtain $4a_6-4a_7\equiv 0 \pmod{8}$, and hence $a_6=a_7$. Similarly, using $\chi_{(4,8)}$ and reducing mod 8 reveals that $a_8=a_9$.

Next we examine $\chi_{(3,0)}(C') = 8 - 8 +4a_6+4a_7-4a_8-4a_9+2a_0+2a_1-2a_2-2a_3-2a_4-2a_5+1+1$. After making the substitutions above and reducing mod 8 we obtain $4a_0-4a_2 \equiv 0 \pmod{8}$, and hence $a_0=a_2$. Since we already knew $a_2=a_3$, we now have $a_0=a_1=a_2=a_3$. 

Applying $\chi_{(3,0)}$ again and reducing mod 16 we obtain $8a_6-8a_8\equiv 8 \pmod{16}$ (recall $r=8$), and hence $a_6\neq a_8$. 
Since we seek only canonical colorings and $a_0=a_1=a_2=a_3$ we must have $a_6=1, a_8=0$. Hence also $a_7=1, a_9=0$. Similarly, we find $a_{10}=1, a_{11}=1, a_{12}=0, a_{13}=0$. 

Next, applying $\chi_{(4,16)}$ and reducing mod 8 reveals that $a_{14}=a_{15}$, while applying $\chi_{(4,22)}$ and reducing mod 8 reveals that $a_{16}\neq a_{17}$, and since we seek only canonical colorings we have $a_{16}=1, a_{17}=0$.

Finally, we have $\chi_{(3,8)}(C') = 8-8 + 4a_{14} + 4a_{15} -4a_{16} -4a_{17} +2a_4+2a_5-2a_0 - 2a_1 -2a_2 -2a_3 +2$. After making the substitutions above and reducing mod 16 we obtain $8 a_{14}-8 a_{0} \equiv 8 \pmod{16}$, and hence $a_0\neq a_{14}$. 

We are now done with step (1) of the algorithm from Section \ref{SecTreeSearch}. Call the coloring that $C'$ has morphed into at this point $Q$. $Q$ is depicted in Figure \ref{FigC16_2}. Note that $a_0\neq a_{14} = a_{15}$ is recorded as $a_{15} = a_{14} = 1-a_{0}$, and that we have reduced the 18 variables from $C'$ to only one independent variable. 

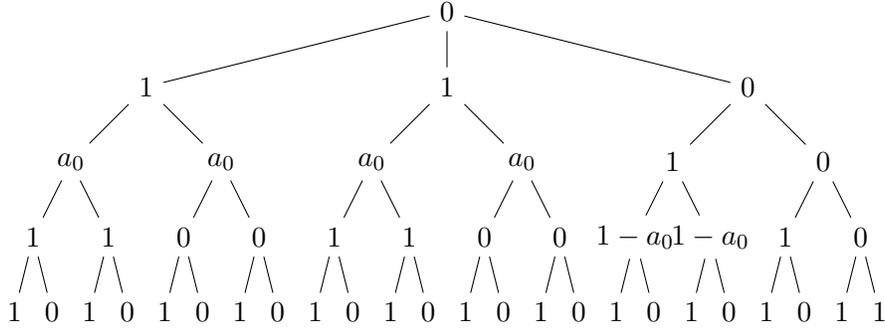
\begin{figure}%
\caption{$C_{16}\times C_{16}$ example, $Q$}%
\label{FigC16_2}%
\begin{tikzpicture}
	\node (00) at (0,0) {$0$};
	\node (10) at (-4,-1) {$1$};
	\node (11) at (0,-1) {$1$};
	\node (12) at (4,-1) {$0$};
	\node (20) at (-5,-2) {$a_0$};
	\node (21) at (-3,-2) {$a_0$};
	\node (22) at (-1,-2) {$a_0$};
	\node (23) at (1,-2) {$a_0$};
	\node (24) at (3,-2) {$1$};
	\node (25) at (5,-2) {$0$};
	\node (30) at (-5.5,-3) {$1$};
	\node (31) at (-4.5,-3) {$1$};
	\node (32) at (-3.5,-3) {$0$};
	\node (33) at (-2.5,-3) {$0$};
	\node (34) at (-1.5,-3) {$1$};
	\node (35) at (-0.5,-3) {$1$};
	\node (36) at (0.5,-3) {$0$};
	\node (37) at (1.5,-3) {$0$};
	\node (38) at (2.5,-3) {$1-a_0$};
	\node (39) at (3.5,-3) {$1-a_0$};
	\node (310) at (4.5,-3) {$1$};
	\node (311) at (5.5,-3) {$0$};
\node (40) at (-5.75,-4) {$1$};
\node (41) at (-5.25,-4) {$0$};
\node (42) at (-4.75,-4) {$1$};
\node (43) at (-4.25,-4) {$0$};
\node (44) at (-3.75,-4) {$1$};
\node (45) at (-3.25,-4) {$0$};
\node (46) at (-2.75,-4) {$1$};
\node (47) at (-2.25,-4) {$0$};
\node (48) at (-1.75,-4) {$1$};
\node (49) at (-1.25,-4) {$0$};
\node (410) at (-0.75,-4) {$1$};
\node (411) at (-0.25,-4) {$0$};
\node (412) at (0.25,-4) {$1$};
\node (413) at (0.75,-4) {$0$};
\node (414) at (1.25,-4) {$1$};
\node (415) at (1.75,-4) {$0$};
\node (416) at (2.25,-4) {$1$};
\node (417) at (2.75,-4) {$0$};
\node (418) at (3.25,-4) {$1$};
\node (419) at (3.75,-4) {$0$};
\node (420) at (4.25,-4) {$1$};
\node (421) at (4.75,-4) {$0$};
\node (422) at (5.25,-4) {$1$};
\node (423) at (5.75,-4) {$1$};
	\draw (00)--(10);
	\draw (00)--(11);
	\draw (00)--(12);
	\draw (10)--(20);
	\draw (10)--(21);
	\draw (11)--(22);
	\draw (11)--(23);
	\draw (12)--(24);
	\draw (12)--(25);
	\draw (20)--(30);
	\draw (20)--(31);
	\draw (21)--(32);
	\draw (21)--(33);
	\draw (22)--(34);
	\draw (22)--(35);
	\draw (23)--(36);
	\draw (23)--(37);
	\draw (24)--(38);
	\draw (24)--(39);
	\draw (25)--(310);
	\draw (25)--(311);
	\draw (30)--(40);
	\draw (30)--(41);
	\draw (31)--(42);
	\draw (31)--(43);
	\draw (32)--(44);
	\draw (32)--(45);
	\draw (33)--(46);
	\draw (33)--(47);
	\draw (34)--(48);
	\draw (34)--(49);
	\draw (35)--(410);
	\draw (35)--(411);
	\draw (36)--(412);
	\draw (36)--(413);
	\draw (37)--(414);
	\draw (37)--(415);
	\draw (38)--(416);
	\draw (38)--(417);
	\draw (39)--(418);
	\draw (39)--(419);
	\draw (310)--(420);
	\draw (310)--(421);
	\draw (311)--(422);
	\draw (311)--(423);
\end{tikzpicture}
\end{figure}

The root of our search tree thus has two children (one for $a_0=1$ and one for $a_0=0$). It turns out that both are PDS. (They are the ninth and tenth PDS in Table \ref{PDSn4} below.)

The example from this section was a search tree consisting of three nodes---the root and its two children. 
The vast majority of our work consisted of reducing the number of variables in the root before creating its children. In our experience, for larger values of $n$ many search trees continue to remain small (having only one or a single-digit number of nodes), although the amount of work required to reduce and simplify the variables involved grows quickly, and some search trees have thousands of nodes.
\appendix
\section{The partial difference sets in $C_{2^n}\times C_{2^n}$ for $2\leq n\leq 9$}
\label{SecOutput}

In this appendix we list the regular nontrivial canonical PDS of even size in $G_n = C_{2^n}\times C_{2^n}$ for $2\leq n\leq 9$. 

We obtained all of our results with a parallelized {\tt Sage} \cite{Sage} implementation of our algorithm, which we ran on a computational server hosted at Sam Houston State University. The server is equipped with four AMD Opteron\textsuperscript{TM} 6272 processors (totaling 64 cores, each running at 2.1 GHz) and 256 GB of RAM. For $n\leq 9$ our enumeration of the regular nontrivial canonical PDS of even size is complete, and took about 5 CPU-years to obtain.
For $n\geq 10$ we halted our implementation of our algorithm after 24 days and 30 days of real runtime (4.2 CPU-years and 5.3 CPU-years), respectively. We suspect that for $n=10$ our data is missing 16 canonical RHPDS of type B, while for $n=11$ it is missing 32 canonical RHPDS and one sporadic canonical PDS. All of the canonical PDS output by our program (including the complete list for $n\leq 9$) are available at the website \cite{OurWebsite}.

To read the tables in this section and at our website, order the elements of the orbit tree $T_n$ according to their position in the linear order $(L(O_1),\sqsubset)$. That is, the first element of $T_n$ is $T_n(0,0)$, the second is $T_n(1,0)$, the third is $T_n(1,1)$, the fourth is $T_n(1,2)$, the fifth is $T_n(2,0)$, etc. Under this ordering the subsets of $T_n$ are in bijection with the binary strings of length $|T_n|$, and the binary strings in our table are the regular nontrivial canonical PDS of even size in $T_n$. For instance, the PDS depicted in Figure \ref{FigC8_3} is encoded as the binary string $0;110;110010;101010101011$, where we have inserted semicolons between the rows of the coloring for readability. Using standard exponential notation for strings we may write $$0;110;110010;101010101011 = 0;1^2 0;1^2 0^2 1 0; (10)^5 1^2,$$ which is the sixth entry in Table \ref{PDSn3}. We note that the entries at our website do {\em not} use exponential notation and hence look like the string on the left instead.

The following PDS in $G_n$ were known before this paper:
\begin{itemize}
	\item All PDS for $n = 2$ \cite{LeifmanMuzychuk95}.
	\item All PDS of the form PCP(2) or PCP(3)$^c$. For each $n>3$ there is one PDS of each of these types. They are the first and last entries (respectively) in each table, and are homogenous \cite{LeifmanMuzychuk95}. 
	\item The second PDS in Table \ref{PDSn3} \cite{StuckeyThesis, HaemersSpence, LeifmanMuzychuk95}.
	\item The second PDS in Table \ref{PDSn4} \cite{GearheartThesis}.
\end{itemize}
To the knowledge of the authors all of the other PDS reported here and at our website are new.

All of the binary strings of the regular canonical PDS of even size in $T_n$ follow, along with their eigenvalues and multiplicities, for $2\leq n \leq 9$.

$n=2$:
\begin{longtable}{|p{1\linewidth}|}
\hline
$k = 6$, $r = 2^{6}$, $s = -2^{9}$ \\
$0$; $0^{3}$; $101010$ \\
\hline
$k = 6$, $r = 2^{6}$, $s = -2^{9}$ \\
$0$; $1^{2}0$; $1010^{3}$ \\
\hline
$k = 10$, $r = 2^{5}$, $s = -2^{10}$ \\
$0$; $1^{2}0$; $10101^{2}$ \\
\hline
\caption{The PDS in $C_{2^n}\times C_{2^n}$, $n=2$}
\label{PDSn2}
\end{longtable}
$n=3$: 
\begin{longtable}{|p{1\linewidth}|}
\hline
$k = 14$, $r = 6^{14}$, $s = -2^{49}$ \\
$0$; $1^{2}0$; $1010^{3}$; $10^{3}10^{7}$ \\
\hline
$k = 18$, $r = 2^{45}$, $s = -6^{18}$ \\
$0$; $0^{3}$; $101010$; $0^{2}(1000)^{2}10$ \\
\hline
$k = 28$, $r = 4^{28}$, $s = -4^{35}$ \\
$0$; $1^{2}0$; $1^{2}0^{2}10$; $(10)^{5}0^{2}$ \\
\hline
$k = 28$, $r = 4^{28}$, $s = -4^{35}$ \\
$0$; $0^{3}$; $1^{2}0^{4}$; $(10)^{6}$ \\
\hline
$k = 36$, $r = 4^{27}$, $s = -4^{36}$ \\
$0$; $0^{3}$; $1^{6}$; $(10)^{6}$ \\
\hline
$k = 36$, $r = 4^{27}$, $s = -4^{36}$ \\
$0$; $1^{2}0$; $1^{2}0^{2}10$; $(10)^{5}1^{2}$ \\
\hline
$k = 42$, $r = 2^{42}$, $s = -6^{21}$ \\
$0$; $1^{2}0$; $10101^{2}$; $1^{3}01^{3}01010$ \\
\hline
$k = 42$, $r = 2^{42}$, $s = -6^{21}$ \\
$0$; $0^{3}$; $101010$; $1^{3}(0111)^{2}0$ \\
\hline
\caption{The PDS in $C_{2^n}\times C_{2^n}$, $n=3$}
\label{PDSn3}
\end{longtable}
$n=4$:
\begin{longtable}{|p{1\linewidth}|}
\hline
$k = 30$, $r = 14^{30}$, $s = -2^{225}$ \\
$0$; $1^{2}0$; $1010^{3}$; $10^{3}10^{7}$; $10^{7}10^{15}$ \\
\hline
$k = 90$, $r = 10^{90}$, $s = -6^{165}$ \\
$0$; $0^{3}$; $101010$; $0^{2}(1000)^{2}10$; $(10101000)^{3}$ \\
\hline
$k = 90$, $r = 10^{90}$, $s = -6^{165}$ \\
$0$; $1^{2}0$; $10101^{2}$; $0^{2}10^{3}101010$; $(10101000)^{2}(1000)^{2}$ \\
\hline
$k = 120$, $r = 8^{120}$, $s = -8^{135}$ \\
$0$; $0^{3}$; $0^{6}$; $1^{2}0^{2}1^{2}0^{2}1^{2}0^{2}$; $(10)^{12}$ \\
\hline
$k = 120$, $r = 8^{120}$, $s = -8^{135}$ \\
$0$; $0^{3}$; $1^{4}0^{2}$; $1^{2}0^{2}1^{2}0^{6}$; $(10)^{12}$ \\
\hline
$k = 120$, $r = 8^{120}$, $s = -8^{135}$ \\
$0$; $1^{2}0$; $0^{4}10$; $1^{2}0^{2}1^{2}0^{2}1^{3}0$; $(10)^{11}0^{2}$ \\
\hline
$k = 120$, $r = 8^{120}$, $s = -8^{135}$ \\
$0$; $1^{2}0$; $1^{5}0$; $1^{2}0^{2}1^{2}0^{4}10$; $(10)^{11}0^{2}$ \\
\hline
$k = 136$, $r = 8^{119}$, $s = -8^{136}$ \\
$0$; $0^{3}$; $1^{4}0^{2}$; $1^{2}0^{2}1^{2}0^{2}1^{4}$; $(10)^{12}$ \\
\hline
$k = 136$, $r = 8^{119}$, $s = -8^{136}$ \\
$0$; $1^{2}0$; $0^{4}10$; $1^{2}0^{2}1^{2}0^{2}1^{3}0$; $(10)^{11}1^{2}$ \\
\hline
$k = 136$, $r = 8^{119}$, $s = -8^{136}$ \\
$0$; $1^{2}0$; $1^{5}0$; $1^{2}0^{2}1^{2}0^{4}10$; $(10)^{11}1^{2}$ \\
\hline
$k = 150$, $r = 6^{150}$, $s = -10^{105}$ \\
$0$; $1^{2}0$; $1010^{3}$; $101^{3}01^{4}0^{2}$; $1^{3}(01)^{3}1^{2}(01)^{6}0$ \\
\hline
$k = 180$, $r = 4^{180}$, $s = -12^{75}$ \\
$0$; $0^{3}$; $1^{6}$; $(10)^{6}$; $(1110)^{6}$ \\
\hline
$k = 210$, $r = 2^{210}$, $s = -14^{45}$ \\
$0$; $0^{3}$; $101010$; $1^{3}(0111)^{2}0$; $1^{7}01^{7}01^{7}0$ \\
\hline
\caption{The PDS in $C_{2^n}\times C_{2^n}$, $n=4$}
\label{PDSn4}
\end{longtable}
$n=5$:
\begin{longtable}{|p{1\linewidth}|}
\hline
$k = 62$, $r = 30^{62}$, $s = -2^{961}$ \\
$0$; $1^{2}0$; $1010^{3}$; $10^{3}10^{7}$; $10^{7}10^{15}$; $10^{15}10^{31}$ \\
\hline
$k = 372$, $r = 20^{372}$, $s = -12^{651}$ \\
$0$; $0^{3}$; $1^{6}$; $(10)^{6}$; $(0010)^{6}$; $(10101000)^{6}$ \\
\hline
$k = 496$, $r = 16^{496}$, $s = -16^{527}$ \\
$0$; $0^{3}$; $0^{6}$; $1^{4}0^{8}$; $(1100)^{6}$; $(10)^{24}$ \\
\hline
$k = 496$, $r = 16^{496}$, $s = -16^{527}$ \\
$0$; $0^{3}$; $1^{4}0^{2}$; $1^{4}0^{4}1^{2}0^{2}$; $(1100)^{5}0^{4}$; $(10)^{24}$ \\
\hline
$k = 496$, $r = 16^{496}$, $s = -16^{527}$ \\
$0$; $1^{2}0$; $0^{4}10$; $1^{4}0^{6}10$; $(1100)^{5}1^{3}0$; $(10)^{23}0^{2}$ \\
\hline
$k = 496$, $r = 16^{496}$, $s = -16^{527}$ \\
$0$; $1^{2}0$; $1^{5}0$; $1^{4}0^{4}1^{3}0$; $(1100)^{5}0^{2}10$; $(10)^{23}0^{2}$ \\
\hline
$k = 528$, $r = 16^{495}$, $s = -16^{528}$ \\
$0$; $0^{3}$; $0^{6}$; $1^{12}$; $(1100)^{6}$; $(10)^{24}$ \\
\hline
$k = 528$, $r = 16^{495}$, $s = -16^{528}$ \\
$0$; $0^{3}$; $1^{4}0^{2}$; $1^{4}0^{4}1^{2}0^{2}$; $(1100)^{5}1^{4}$; $(10)^{24}$ \\
\hline
$k = 528$, $r = 16^{495}$, $s = -16^{528}$ \\
$0$; $1^{2}0$; $0^{4}10$; $1^{4}0^{6}10$; $(1100)^{5}1^{3}0$; $(10)^{23}1^{2}$ \\
\hline
$k = 528$, $r = 16^{495}$, $s = -16^{528}$ \\
$0$; $1^{2}0$; $1^{5}0$; $1^{4}0^{4}1^{3}0$; $(1100)^{5}0^{2}10$; $(10)^{23}1^{2}$ \\
\hline
$k = 558$, $r = 14^{558}$, $s = -18^{465}$ \\
$0$; $1^{2}0$; $1010^{3}$; $(1000)^{2}1^{4}$; $(10111100)^{2}1^{2}0^{2}1^{2}0^{2}$; $1^{3}(01)^{7}1^{2}(01)^{14}0$ \\
\hline
$k = 930$, $r = 2^{930}$, $s = -30^{93}$ \\
$0$; $0^{3}$; $101010$; $1^{3}(0111)^{2}0$; $1^{7}01^{7}01^{7}0$; $1^{15}01^{15}01^{15}0$ \\
\hline
\caption{The PDS in $C_{2^n}\times C_{2^n}$, $n=5$}
\label{PDSn5}
\end{longtable}
$n=6$:
\begin{longtable}{|p{1\linewidth}|}
\hline
$k = 126$, $r = 62^{126}$, $s = -2^{3969}$ \\
$0$; $1^{2}0$; $1010^{3}$; $10^{3}10^{7}$; $10^{7}10^{15}$; $10^{15}10^{31}$; $10^{31}10^{63}$ \\
\hline
$k = 1890$, $r = 34^{1890}$, $s = -30^{2205}$ \\
$0$; $0^{3}$; $101010$; $1^{3}(0111)^{2}0$; $0^{4}1^{3}0^{5}1^{3}0^{5}1^{3}0$; $(1100110011000010)^{3}$; $(10)^{15}0^{2}(10)^{15}0^{2}(10)^{15}0^{2}$ \\
\hline
$k = 2016$, $r = 32^{2016}$, $s = -32^{2079}$ \\
$0$; $0^{3}$; $0^{6}$; $0^{12}$; $1^{4}0^{4}1^{4}0^{4}1^{4}0^{4}$; $(1100)^{12}$; $(10)^{48}$ \\
\hline
$k = 2016$, $r = 32^{2016}$, $s = -32^{2079}$ \\
$0$; $0^{3}$; $0^{6}$; $1^{8}0^{4}$; $1^{4}0^{4}1^{4}0^{12}$; $(1100)^{12}$; $(10)^{48}$ \\
\hline
$k = 2016$, $r = 32^{2016}$, $s = -32^{2079}$ \\
$0$; $0^{3}$; $1^{4}0^{2}$; $0^{8}1^{2}0^{2}$; $1^{4}0^{4}1^{4}0^{4}1^{6}0^{2}$; $(1100)^{11}0^{4}$; $(10)^{48}$ \\
\hline
$k = 2016$, $r = 32^{2016}$, $s = -32^{2079}$ \\
$0$; $0^{3}$; $1^{4}0^{2}$; $1^{10}0^{2}$; $1^{4}0^{4}1^{4}0^{8}1^{2}0^{2}$; $(1100)^{11}0^{4}$; $(10)^{48}$ \\
\hline
$k = 2016$, $r = 32^{2016}$, $s = -32^{2079}$ \\
$0$; $1^{2}0$; $0^{4}10$; $0^{10}10$; $1^{4}0^{4}1^{4}0^{4}1^{4}0^{2}10$; $(1100)^{11}1^{3}0$; $(10)^{47}0^{2}$ \\
\hline
$k = 2016$, $r = 32^{2016}$, $s = -32^{2079}$ \\
$0$; $1^{2}0$; $0^{4}10$; $1^{8}0^{2}10$; $1^{4}0^{4}1^{4}0^{10}10$; $(1100)^{11}1^{3}0$; $(10)^{47}0^{2}$ \\
\hline
$k = 2016$, $r = 32^{2016}$, $s = -32^{2079}$ \\
$0$; $1^{2}0$; $1^{5}0$; $0^{8}1^{3}0$; $1^{4}0^{4}1^{4}0^{4}1^{7}0$; $(1100)^{11}0^{2}10$; $(10)^{47}0^{2}$ \\
\hline
$k = 2016$, $r = 32^{2016}$, $s = -32^{2079}$ \\
$0$; $1^{2}0$; $1^{5}0$; $1^{11}0$; $1^{4}0^{4}1^{4}0^{8}1^{3}0$; $(1100)^{11}0^{2}10$; $(10)^{47}0^{2}$ \\
\hline
$k = 2080$, $r = 32^{2015}$, $s = -32^{2080}$ \\
$0$; $0^{3}$; $0^{6}$; $1^{8}0^{4}$; $1^{4}0^{4}1^{4}0^{4}1^{8}$; $(1100)^{12}$; $(10)^{48}$ \\
\hline
$k = 2080$, $r = 32^{2015}$, $s = -32^{2080}$ \\
$0$; $0^{3}$; $1^{4}0^{2}$; $0^{8}1^{2}0^{2}$; $1^{4}0^{4}1^{4}0^{4}1^{6}0^{2}$; $(1100)^{11}1^{4}$; $(10)^{48}$ \\
\hline
$k = 2080$, $r = 32^{2015}$, $s = -32^{2080}$ \\
$0$; $0^{3}$; $1^{4}0^{2}$; $1^{10}0^{2}$; $1^{4}0^{4}1^{4}0^{8}1^{2}0^{2}$; $(1100)^{11}1^{4}$; $(10)^{48}$ \\
\hline
$k = 2080$, $r = 32^{2015}$, $s = -32^{2080}$ \\
$0$; $1^{2}0$; $0^{4}10$; $0^{10}10$; $1^{4}0^{4}1^{4}0^{4}1^{4}0^{2}10$; $(1100)^{11}1^{3}0$; $(10)^{47}1^{2}$ \\
\hline
$k = 2080$, $r = 32^{2015}$, $s = -32^{2080}$ \\
$0$; $1^{2}0$; $0^{4}10$; $1^{8}0^{2}10$; $1^{4}0^{4}1^{4}0^{10}10$; $(1100)^{11}1^{3}0$; $(10)^{47}1^{2}$ \\
\hline
$k = 2080$, $r = 32^{2015}$, $s = -32^{2080}$ \\
$0$; $1^{2}0$; $1^{5}0$; $0^{8}1^{3}0$; $1^{4}0^{4}1^{4}0^{4}1^{7}0$; $(1100)^{11}0^{2}10$; $(10)^{47}1^{2}$ \\
\hline
$k = 2080$, $r = 32^{2015}$, $s = -32^{2080}$ \\
$0$; $1^{2}0$; $1^{5}0$; $1^{11}0$; $1^{4}0^{4}1^{4}0^{8}1^{3}0$; $(1100)^{11}0^{2}10$; $(10)^{47}1^{2}$ \\
\hline
$k = 2142$, $r = 30^{2142}$, $s = -34^{1953}$ \\
$0$; $1^{2}0$; $1010^{3}$; $10^{3}10^{7}$; $(10001111)^{2}1^{4}0^{4}$; $(1011110011001100)^{2}(1100)^{4}$; $1^{3}(01)^{15}1^{2}(01)^{30}0$ \\
\hline
$k = 3906$, $r = 2^{3906}$, $s = -62^{189}$ \\
$0$; $0^{3}$; $101010$; $1^{3}(0111)^{2}0$; $1^{7}01^{7}01^{7}0$; $1^{15}01^{15}01^{15}0$; $1^{31}01^{31}01^{31}0$ \\
\hline
\caption{The PDS in $C_{2^n}\times C_{2^n}$, $n=6$}
\label{PDSn6}
\end{longtable}
$n=7$:
\begin{longtable}{|p{1\linewidth}|}
\hline
$k = 254$, $r = 126^{254}$, $s = -2^{16129}$ \\
$0$; $1^{2}0$; $1010^{3}$; $10^{3}10^{7}$; $10^{7}10^{15}$; $10^{15}10^{31}$; $10^{31}10^{63}$; $10^{63}10^{127}$ \\
\hline
$k = 8128$, $r = 64^{8128}$, $s = -64^{8255}$ \\
$0$; $0^{3}$; $0^{6}$; $0^{12}$; $1^{8}0^{16}$; $(11110000)^{6}$; $(1100)^{24}$; $(10)^{96}$ \\
\hline
$k = 8128$, $r = 64^{8128}$, $s = -64^{8255}$ \\
$0$; $0^{3}$; $0^{6}$; $1^{8}0^{4}$; $1^{8}0^{8}1^{4}0^{4}$; $1^{4}0^{4}1^{4}0^{4}1^{4}0^{4}1^{4}0^{4}1^{4}0^{12}$; $(1100)^{24}$; $(10)^{96}$ \\
\hline
$k = 8128$, $r = 64^{8128}$, $s = -64^{8255}$ \\
$0$; $0^{3}$; $1^{4}0^{2}$; $0^{8}1^{2}0^{2}$; $1^{8}0^{12}1^{2}0^{2}$; $1^{4}0^{4}1^{4}0^{4}1^{4}0^{4}1^{4}0^{4}1^{4}0^{4}1^{6}0^{2}$; $(1100)^{23}0^{4}$; $(10)^{96}$ \\
\hline
$k = 8128$, $r = 64^{8128}$, $s = -64^{8255}$ \\
$0$; $0^{3}$; $1^{4}0^{2}$; $1^{10}0^{2}$; $1^{8}0^{8}1^{6}0^{2}$; $1^{4}0^{4}1^{4}0^{4}1^{4}0^{4}1^{4}0^{4}1^{4}0^{8}1^{2}0^{2}$; $(1100)^{23}0^{4}$; $(10)^{96}$ \\
\hline
$k = 8128$, $r = 64^{8128}$, $s = -64^{8255}$ \\
$0$; $1^{2}0$; $0^{4}10$; $0^{10}10$; $1^{8}0^{14}10$; $1^{4}0^{4}1^{4}0^{4}1^{4}0^{4}1^{4}0^{4}1^{4}0^{4}1^{4}0^{2}10$; $(1100)^{23}1^{3}0$; $(10)^{95}0^{2}$ \\
\hline
$k = 8128$, $r = 64^{8128}$, $s = -64^{8255}$ \\
$0$; $1^{2}0$; $0^{4}10$; $1^{8}0^{2}10$; $1^{8}0^{8}1^{4}0^{2}10$; $1^{4}0^{4}1^{4}0^{4}1^{4}0^{4}1^{4}0^{4}1^{4}0^{10}10$; $(1100)^{23}1^{3}0$; $(10)^{95}0^{2}$ \\
\hline
$k = 8128$, $r = 64^{8128}$, $s = -64^{8255}$ \\
$0$; $1^{2}0$; $1^{5}0$; $0^{8}1^{3}0$; $1^{8}0^{12}1^{3}0$; $1^{4}0^{4}1^{4}0^{4}1^{4}0^{4}1^{4}0^{4}1^{4}0^{4}1^{7}0$; $(1100)^{23}0^{2}10$; $(10)^{95}0^{2}$ \\
\hline
$k = 8128$, $r = 64^{8128}$, $s = -64^{8255}$ \\
$0$; $1^{2}0$; $1^{5}0$; $1^{11}0$; $1^{8}0^{8}1^{7}0$; $1^{4}0^{4}1^{4}0^{4}1^{4}0^{4}1^{4}0^{4}1^{4}0^{8}1^{3}0$; $(1100)^{23}0^{2}10$; $(10)^{95}0^{2}$ \\
\hline
$k = 8256$, $r = 64^{8127}$, $s = -64^{8256}$ \\
$0$; $0^{3}$; $0^{6}$; $0^{12}$; $1^{24}$; $(11110000)^{6}$; $(1100)^{24}$; $(10)^{96}$ \\
\hline
$k = 8256$, $r = 64^{8127}$, $s = -64^{8256}$ \\
$0$; $0^{3}$; $0^{6}$; $1^{8}0^{4}$; $1^{8}0^{8}1^{4}0^{4}$; $1^{4}0^{4}1^{4}0^{4}1^{4}0^{4}1^{4}0^{4}1^{4}0^{4}1^{8}$; $(1100)^{24}$; $(10)^{96}$ \\
\hline
$k = 8256$, $r = 64^{8127}$, $s = -64^{8256}$ \\
$0$; $0^{3}$; $1^{4}0^{2}$; $0^{8}1^{2}0^{2}$; $1^{8}0^{12}1^{2}0^{2}$; $1^{4}0^{4}1^{4}0^{4}1^{4}0^{4}1^{4}0^{4}1^{4}0^{4}1^{6}0^{2}$; $(1100)^{23}1^{4}$; $(10)^{96}$ \\
\hline
$k = 8256$, $r = 64^{8127}$, $s = -64^{8256}$ \\
$0$; $0^{3}$; $1^{4}0^{2}$; $1^{10}0^{2}$; $1^{8}0^{8}1^{6}0^{2}$; $1^{4}0^{4}1^{4}0^{4}1^{4}0^{4}1^{4}0^{4}1^{4}0^{8}1^{2}0^{2}$; $(1100)^{23}1^{4}$; $(10)^{96}$ \\
\hline
$k = 8256$, $r = 64^{8127}$, $s = -64^{8256}$ \\
$0$; $1^{2}0$; $0^{4}10$; $0^{10}10$; $1^{8}0^{14}10$; $1^{4}0^{4}1^{4}0^{4}1^{4}0^{4}1^{4}0^{4}1^{4}0^{4}1^{4}0^{2}10$; $(1100)^{23}1^{3}0$; $(10)^{95}1^{2}$ \\
\hline
$k = 8256$, $r = 64^{8127}$, $s = -64^{8256}$ \\
$0$; $1^{2}0$; $0^{4}10$; $1^{8}0^{2}10$; $1^{8}0^{8}1^{4}0^{2}10$; $1^{4}0^{4}1^{4}0^{4}1^{4}0^{4}1^{4}0^{4}1^{4}0^{10}10$; $(1100)^{23}1^{3}0$; $(10)^{95}1^{2}$ \\
\hline
$k = 8256$, $r = 64^{8127}$, $s = -64^{8256}$ \\
$0$; $1^{2}0$; $1^{5}0$; $0^{8}1^{3}0$; $1^{8}0^{12}1^{3}0$; $1^{4}0^{4}1^{4}0^{4}1^{4}0^{4}1^{4}0^{4}1^{4}0^{4}1^{7}0$; $(1100)^{23}0^{2}10$; $(10)^{95}1^{2}$ \\
\hline
$k = 8256$, $r = 64^{8127}$, $s = -64^{8256}$ \\
$0$; $1^{2}0$; $1^{5}0$; $1^{11}0$; $1^{8}0^{8}1^{7}0$; $1^{4}0^{4}1^{4}0^{4}1^{4}0^{4}1^{4}0^{4}1^{4}0^{8}1^{3}0$; $(1100)^{23}0^{2}10$; $(10)^{95}1^{2}$ \\
\hline
$k = 8382$, $r = 62^{8382}$, $s = -66^{8001}$ \\
$0$; $1^{2}0$; $1010^{3}$; $10^{3}10^{7}$; $10^{7}10^{7}1^{8}$; $(1000111111110000)^{2}1^{4}0^{4}1^{4}0^{4}$; $(10111100110011001100110011001100)^{2}(1100)^{8}$; $1^{3}(01)^{31}1^{2}(01)^{62}0$ \\
\hline
$k = 16002$, $r = 2^{16002}$, $s = -126^{381}$ \\
$0$; $0^{3}$; $101010$; $1^{3}(0111)^{2}0$; $1^{7}01^{7}01^{7}0$; $1^{15}01^{15}01^{15}0$; $1^{31}01^{31}01^{31}0$; $1^{63}01^{63}01^{63}0$ \\
\hline
\caption{The PDS in $C_{2^n}\times C_{2^n}$, $n=7$}
\label{PDSn7}
\end{longtable}
$n=8$:
\begin{longtable}{|p{1\linewidth}|}
\hline
$k = 510$, $r = 254^{510}$, $s = -2^{65025}$ \\
$0$; $1^{2}0$; $1010^{3}$; $10^{3}10^{7}$; $10^{7}10^{15}$; $10^{15}10^{31}$; $10^{31}10^{63}$; $10^{63}10^{127}$; $10^{127}10^{255}$ \\
\hline
$k = 32130$, $r = 130^{32130}$, $s = -126^{33405}$ \\
$0$; $0^{3}$; $101010$; $1^{3}(0111)^{2}0$; $1^{7}01^{7}01^{7}0$; $0^{8}1^{7}0^{9}1^{7}0^{9}1^{7}0$;\\$1^{4}0^{4}1^{4}0^{4}1^{4}0^{8}1^{3}(01111000)^{3}0^{5}1^{3}(01111000)^{3}0^{5}1^{3}0$;\\$(1100)^{15}0^{2}(1011001100110011001100110011001100110011001100110011001100110000)^{2}10$; $(10)^{63}0^{2}(10)^{63}0^{2}(10)^{63}0^{2}$ \\
\hline
$k = 32640$, $r = 128^{32640}$, $s = -128^{32895}$ \\
$0$; $0^{3}$; $0^{6}$; $0^{12}$; $0^{24}$; $1^{8}0^{8}1^{8}0^{8}1^{8}0^{8}$; $(11110000)^{12}$; $(1100)^{48}$; $(10)^{192}$ \\
\hline
$k = 32640$, $r = 128^{32640}$, $s = -128^{32895}$ \\
$0$; $0^{3}$; $0^{6}$; $0^{12}$; $1^{16}0^{8}$; $1^{8}0^{8}1^{8}0^{24}$; $(11110000)^{12}$; $(1100)^{48}$; $(10)^{192}$ \\
\hline
$k = 32640$, $r = 128^{32640}$, $s = -128^{32895}$ \\
$0$; $0^{3}$; $0^{6}$; $1^{8}0^{4}$; $0^{16}1^{4}0^{4}$; $1^{8}0^{8}1^{8}0^{8}1^{12}0^{4}$; $(11110000)^{11}0^{8}$; $(1100)^{48}$; $(10)^{192}$ \\
\hline
$k = 32640$, $r = 128^{32640}$, $s = -128^{32895}$ \\
$0$; $0^{3}$; $0^{6}$; $1^{8}0^{4}$; $1^{20}0^{4}$; $1^{8}0^{8}1^{8}0^{16}1^{4}0^{4}$; $(11110000)^{11}0^{8}$; $(1100)^{48}$; $(10)^{192}$ \\
\hline
$k = 32640$, $r = 128^{32640}$, $s = -128^{32895}$ \\
$0$; $0^{3}$; $1^{4}0^{2}$; $0^{8}1^{2}0^{2}$; $0^{20}1^{2}0^{2}$; $1^{8}0^{8}1^{8}0^{8}1^{8}0^{4}1^{2}0^{2}$; $(11110000)^{11}1^{6}0^{2}$; $(1100)^{47}0^{4}$; $(10)^{192}$ \\
\hline
$k = 32640$, $r = 128^{32640}$, $s = -128^{32895}$ \\
$0$; $0^{3}$; $1^{4}0^{2}$; $0^{8}1^{2}0^{2}$; $1^{16}0^{4}1^{2}0^{2}$; $1^{8}0^{8}1^{8}0^{20}1^{2}0^{2}$; $(11110000)^{11}1^{6}0^{2}$; $(1100)^{47}0^{4}$; $(10)^{192}$ \\
\hline
$k = 32640$, $r = 128^{32640}$, $s = -128^{32895}$ \\
$0$; $0^{3}$; $1^{4}0^{2}$; $1^{10}0^{2}$; $0^{16}1^{6}0^{2}$; $1^{8}0^{8}1^{8}0^{8}1^{14}0^{2}$; $(11110000)^{11}0^{4}1^{2}0^{2}$; $(1100)^{47}0^{4}$; $(10)^{192}$ \\
\hline
$k = 32640$, $r = 128^{32640}$, $s = -128^{32895}$ \\
$0$; $0^{3}$; $1^{4}0^{2}$; $1^{10}0^{2}$; $1^{22}0^{2}$; $1^{8}0^{8}1^{8}0^{16}1^{6}0^{2}$; $(11110000)^{11}0^{4}1^{2}0^{2}$; $(1100)^{47}0^{4}$; $(10)^{192}$ \\
\hline
$k = 32640$, $r = 128^{32640}$, $s = -128^{32895}$ \\
$0$; $1^{2}0$; $0^{4}10$; $0^{10}10$; $0^{22}10$; $1^{8}0^{8}1^{8}0^{8}1^{8}0^{6}10$; $(11110000)^{11}1^{4}0^{2}10$; $(1100)^{47}1^{3}0$; $(10)^{191}0^{2}$ \\
\hline
$k = 32640$, $r = 128^{32640}$, $s = -128^{32895}$ \\
$0$; $1^{2}0$; $0^{4}10$; $0^{10}10$; $1^{16}0^{6}10$; $1^{8}0^{8}1^{8}0^{22}10$; $(11110000)^{11}1^{4}0^{2}10$; $(1100)^{47}1^{3}0$; $(10)^{191}0^{2}$ \\
\hline
$k = 32640$, $r = 128^{32640}$, $s = -128^{32895}$ \\
$0$; $1^{2}0$; $0^{4}10$; $1^{8}0^{2}10$; $0^{16}1^{4}0^{2}10$; $1^{8}0^{8}1^{8}0^{8}1^{12}0^{2}10$; $(11110000)^{11}0^{6}10$; $(1100)^{47}1^{3}0$; $(10)^{191}0^{2}$ \\
\hline
$k = 32640$, $r = 128^{32640}$, $s = -128^{32895}$ \\
$0$; $1^{2}0$; $0^{4}10$; $1^{8}0^{2}10$; $1^{20}0^{2}10$; $1^{8}0^{8}1^{8}0^{16}1^{4}0^{2}10$; $(11110000)^{11}0^{6}10$; $(1100)^{47}1^{3}0$; $(10)^{191}0^{2}$ \\
\hline
$k = 32640$, $r = 128^{32640}$, $s = -128^{32895}$ \\
$0$; $1^{2}0$; $1^{5}0$; $0^{8}1^{3}0$; $0^{20}1^{3}0$; $1^{8}0^{8}1^{8}0^{8}1^{8}0^{4}1^{3}0$; $(11110000)^{11}1^{7}0$; $(1100)^{47}0^{2}10$; $(10)^{191}0^{2}$ \\
\hline
$k = 32640$, $r = 128^{32640}$, $s = -128^{32895}$ \\
$0$; $1^{2}0$; $1^{5}0$; $0^{8}1^{3}0$; $1^{16}0^{4}1^{3}0$; $1^{8}0^{8}1^{8}0^{20}1^{3}0$; $(11110000)^{11}1^{7}0$; $(1100)^{47}0^{2}10$; $(10)^{191}0^{2}$ \\
\hline
$k = 32640$, $r = 128^{32640}$, $s = -128^{32895}$ \\
$0$; $1^{2}0$; $1^{5}0$; $1^{11}0$; $0^{16}1^{7}0$; $1^{8}0^{8}1^{8}0^{8}1^{15}0$; $(11110000)^{11}0^{4}1^{3}0$; $(1100)^{47}0^{2}10$; $(10)^{191}0^{2}$ \\
\hline
$k = 32640$, $r = 128^{32640}$, $s = -128^{32895}$ \\
$0$; $1^{2}0$; $1^{5}0$; $1^{11}0$; $1^{23}0$; $1^{8}0^{8}1^{8}0^{16}1^{7}0$; $(11110000)^{11}0^{4}1^{3}0$; $(1100)^{47}0^{2}10$; $(10)^{191}0^{2}$ \\
\hline
$k = 32896$, $r = 128^{32639}$, $s = -128^{32896}$ \\
$0$; $0^{3}$; $0^{6}$; $0^{12}$; $1^{16}0^{8}$; $1^{8}0^{8}1^{8}0^{8}1^{16}$; $(11110000)^{12}$; $(1100)^{48}$; $(10)^{192}$ \\
\hline
$k = 32896$, $r = 128^{32639}$, $s = -128^{32896}$ \\
$0$; $0^{3}$; $0^{6}$; $1^{8}0^{4}$; $0^{16}1^{4}0^{4}$; $1^{8}0^{8}1^{8}0^{8}1^{12}0^{4}$; $(11110000)^{11}1^{8}$; $(1100)^{48}$; $(10)^{192}$ \\
\hline
$k = 32896$, $r = 128^{32639}$, $s = -128^{32896}$ \\
$0$; $0^{3}$; $0^{6}$; $1^{8}0^{4}$; $1^{20}0^{4}$; $1^{8}0^{8}1^{8}0^{16}1^{4}0^{4}$; $(11110000)^{11}1^{8}$; $(1100)^{48}$; $(10)^{192}$ \\
\hline
$k = 32896$, $r = 128^{32639}$, $s = -128^{32896}$ \\
$0$; $0^{3}$; $1^{4}0^{2}$; $0^{8}1^{2}0^{2}$; $0^{20}1^{2}0^{2}$; $1^{8}0^{8}1^{8}0^{8}1^{8}0^{4}1^{2}0^{2}$; $(11110000)^{11}1^{6}0^{2}$; $(1100)^{47}1^{4}$; $(10)^{192}$ \\
\hline
$k = 32896$, $r = 128^{32639}$, $s = -128^{32896}$ \\
$0$; $0^{3}$; $1^{4}0^{2}$; $0^{8}1^{2}0^{2}$; $1^{16}0^{4}1^{2}0^{2}$; $1^{8}0^{8}1^{8}0^{20}1^{2}0^{2}$; $(11110000)^{11}1^{6}0^{2}$; $(1100)^{47}1^{4}$; $(10)^{192}$ \\
\hline
$k = 32896$, $r = 128^{32639}$, $s = -128^{32896}$ \\
$0$; $0^{3}$; $1^{4}0^{2}$; $1^{10}0^{2}$; $0^{16}1^{6}0^{2}$; $1^{8}0^{8}1^{8}0^{8}1^{14}0^{2}$; $(11110000)^{11}0^{4}1^{2}0^{2}$; $(1100)^{47}1^{4}$; $(10)^{192}$ \\
\hline
$k = 32896$, $r = 128^{32639}$, $s = -128^{32896}$ \\
$0$; $0^{3}$; $1^{4}0^{2}$; $1^{10}0^{2}$; $1^{22}0^{2}$; $1^{8}0^{8}1^{8}0^{16}1^{6}0^{2}$; $(11110000)^{11}0^{4}1^{2}0^{2}$; $(1100)^{47}1^{4}$; $(10)^{192}$ \\
\hline
$k = 32896$, $r = 128^{32639}$, $s = -128^{32896}$ \\
$0$; $1^{2}0$; $0^{4}10$; $0^{10}10$; $0^{22}10$; $1^{8}0^{8}1^{8}0^{8}1^{8}0^{6}10$; $(11110000)^{11}1^{4}0^{2}10$; $(1100)^{47}1^{3}0$; $(10)^{191}1^{2}$ \\
\hline
$k = 32896$, $r = 128^{32639}$, $s = -128^{32896}$ \\
$0$; $1^{2}0$; $0^{4}10$; $0^{10}10$; $1^{16}0^{6}10$; $1^{8}0^{8}1^{8}0^{22}10$; $(11110000)^{11}1^{4}0^{2}10$; $(1100)^{47}1^{3}0$; $(10)^{191}1^{2}$ \\
\hline
$k = 32896$, $r = 128^{32639}$, $s = -128^{32896}$ \\
$0$; $1^{2}0$; $0^{4}10$; $1^{8}0^{2}10$; $0^{16}1^{4}0^{2}10$; $1^{8}0^{8}1^{8}0^{8}1^{12}0^{2}10$; $(11110000)^{11}0^{6}10$; $(1100)^{47}1^{3}0$; $(10)^{191}1^{2}$ \\
\hline
$k = 32896$, $r = 128^{32639}$, $s = -128^{32896}$ \\
$0$; $1^{2}0$; $0^{4}10$; $1^{8}0^{2}10$; $1^{20}0^{2}10$; $1^{8}0^{8}1^{8}0^{16}1^{4}0^{2}10$; $(11110000)^{11}0^{6}10$; $(1100)^{47}1^{3}0$; $(10)^{191}1^{2}$ \\
\hline
$k = 32896$, $r = 128^{32639}$, $s = -128^{32896}$ \\
$0$; $1^{2}0$; $1^{5}0$; $0^{8}1^{3}0$; $0^{20}1^{3}0$; $1^{8}0^{8}1^{8}0^{8}1^{8}0^{4}1^{3}0$; $(11110000)^{11}1^{7}0$; $(1100)^{47}0^{2}10$; $(10)^{191}1^{2}$ \\
\hline
$k = 32896$, $r = 128^{32639}$, $s = -128^{32896}$ \\
$0$; $1^{2}0$; $1^{5}0$; $0^{8}1^{3}0$; $1^{16}0^{4}1^{3}0$; $1^{8}0^{8}1^{8}0^{20}1^{3}0$; $(11110000)^{11}1^{7}0$; $(1100)^{47}0^{2}10$; $(10)^{191}1^{2}$ \\
\hline
$k = 32896$, $r = 128^{32639}$, $s = -128^{32896}$ \\
$0$; $1^{2}0$; $1^{5}0$; $1^{11}0$; $0^{16}1^{7}0$; $1^{8}0^{8}1^{8}0^{8}1^{15}0$; $(11110000)^{11}0^{4}1^{3}0$; $(1100)^{47}0^{2}10$; $(10)^{191}1^{2}$ \\
\hline
$k = 32896$, $r = 128^{32639}$, $s = -128^{32896}$ \\
$0$; $1^{2}0$; $1^{5}0$; $1^{11}0$; $1^{23}0$; $1^{8}0^{8}1^{8}0^{16}1^{7}0$; $(11110000)^{11}0^{4}1^{3}0$; $(1100)^{47}0^{2}10$; $(10)^{191}1^{2}$ \\
\hline
$k = 33150$, $r = 126^{33150}$, $s = -130^{32385}$ \\
$0$; $1^{2}0$; $1010^{3}$; $10^{3}10^{7}$; $10^{7}10^{15}$; $10^{7}1^{9}0^{7}1^{16}0^{8}$; \\$(10001111111100001111000011110000)^{2}1^{4}0^{4}1^{4}0^{4}1^{4}0^{4}1^{4}0^{4}$; \\$(1011110011001100110011001100110011001100110011001100110011001100)^{2}(1100)^{16}$; $1^{3}(01)^{63}1^{2}(01)^{126}0$ \\
\hline
$k = 64770$, $r = 2^{64770}$, $s = -254^{765}$ \\
$0$; $0^{3}$; $101010$; $1^{3}(0111)^{2}0$; $1^{7}01^{7}01^{7}0$; $1^{15}01^{15}01^{15}0$; $1^{31}01^{31}01^{31}0$; $1^{63}01^{63}01^{63}0$; $1^{127}01^{127}01^{127}0$ \\
\hline
\caption{The PDS in $C_{2^n}\times C_{2^n}$, $n=8$}
\label{PDSn8}
\end{longtable}
$n=9$:
\begin{longtable}{|p{1\linewidth}|}
\hline
$k = 1022$, $r = 510^{1022}$, $s = -2^{261121}$ \\
$0$; $1^{2}0$; $1010^{3}$; $10^{3}10^{7}$; $10^{7}10^{15}$; $10^{15}10^{31}$; $10^{31}10^{63}$; $10^{63}10^{127}$; $10^{127}10^{255}$; $10^{255}10^{511}$ \\
\hline
$k = 130816$, $r = 256^{130816}$, $s = -256^{131327}$ \\
$0$; $0^{3}$; $1^{4}0^{2}$; $0^{8}1^{2}0^{2}$; $1^{16}0^{4}1^{2}0^{2}$; $1^{16}0^{16}1^{8}0^{4}1^{2}0^{2}$; $1^{8}0^{8}1^{8}0^{8}1^{8}0^{8}1^{8}0^{8}1^{8}0^{20}1^{2}0^{2}$; $(11110000)^{23}1^{6}0^{2}$; $(1100)^{95}0^{4}$; $(10)^{384}$ \\
\hline
$k = 130816$, $r = 256^{130816}$, $s = -256^{131327}$ \\
$0$; $0^{3}$; $1^{4}0^{2}$; $0^{8}1^{2}0^{2}$; $0^{20}1^{2}0^{2}$; $1^{16}0^{28}1^{2}0^{2}$; $1^{8}0^{8}1^{8}0^{8}1^{8}0^{8}1^{8}0^{8}1^{8}0^{8}1^{8}0^{4}1^{2}0^{2}$; $(11110000)^{23}1^{6}0^{2}$; $(1100)^{95}0^{4}$; $(10)^{384}$ \\
\hline
$k = 130816$, $r = 256^{130816}$, $s = -256^{131327}$ \\
$0$; $0^{3}$; $1^{4}0^{2}$; $1^{10}0^{2}$; $0^{16}1^{6}0^{2}$; $1^{16}0^{24}1^{6}0^{2}$; $1^{8}0^{8}1^{8}0^{8}1^{8}0^{8}1^{8}0^{8}1^{8}0^{8}1^{14}0^{2}$; $(11110000)^{23}0^{4}1^{2}0^{2}$; $(1100)^{95}0^{4}$; $(10)^{384}$ \\
\hline
$k = 130816$, $r = 256^{130816}$, $s = -256^{131327}$ \\
$0$; $0^{3}$; $1^{4}0^{2}$; $1^{10}0^{2}$; $1^{22}0^{2}$; $1^{16}0^{16}1^{14}0^{2}$; $1^{8}0^{8}1^{8}0^{8}1^{8}0^{8}1^{8}0^{8}1^{8}0^{16}1^{6}0^{2}$; $(11110000)^{23}0^{4}1^{2}0^{2}$; $(1100)^{95}0^{4}$; $(10)^{384}$ \\
\hline
$k = 130816$, $r = 256^{130816}$, $s = -256^{131327}$ \\
$0$; $0^{3}$; $0^{6}$; $0^{12}$; $1^{16}0^{8}$; $1^{16}0^{16}1^{8}0^{8}$; $1^{8}0^{8}1^{8}0^{8}1^{8}0^{8}1^{8}0^{8}1^{8}0^{24}$; $(11110000)^{24}$; $(1100)^{96}$; $(10)^{384}$ \\
\hline
$k = 130816$, $r = 256^{130816}$, $s = -256^{131327}$ \\
$0$; $0^{3}$; $0^{6}$; $1^{8}0^{4}$; $0^{16}1^{4}0^{4}$; $1^{16}0^{24}1^{4}0^{4}$; $1^{8}0^{8}1^{8}0^{8}1^{8}0^{8}1^{8}0^{8}1^{8}0^{8}1^{12}0^{4}$; $(11110000)^{23}0^{8}$; $(1100)^{96}$; $(10)^{384}$ \\
\hline
$k = 130816$, $r = 256^{130816}$, $s = -256^{131327}$ \\
$0$; $0^{3}$; $0^{6}$; $1^{8}0^{4}$; $1^{20}0^{4}$; $1^{16}0^{16}1^{12}0^{4}$; $1^{8}0^{8}1^{8}0^{8}1^{8}0^{8}1^{8}0^{8}1^{8}0^{16}1^{4}0^{4}$; $(11110000)^{23}0^{8}$; $(1100)^{96}$; $(10)^{384}$ \\
\hline
$k = 130816$, $r = 256^{130816}$, $s = -256^{131327}$ \\
$0$; $0^{3}$; $0^{6}$; $0^{12}$; $0^{24}$; $1^{16}0^{32}$; $1^{8}0^{8}1^{8}0^{8}1^{8}0^{8}1^{8}0^{8}1^{8}0^{8}1^{8}0^{8}$; $(11110000)^{24}$; $(1100)^{96}$; $(10)^{384}$ \\
\hline
$k = 130816$, $r = 256^{130816}$, $s = -256^{131327}$ \\
$0$; $1^{2}0$; $0^{4}10$; $0^{10}10$; $0^{22}10$; $1^{16}0^{30}10$; $1^{8}0^{8}1^{8}0^{8}1^{8}0^{8}1^{8}0^{8}1^{8}0^{8}1^{8}0^{6}10$; $(11110000)^{23}1^{4}0^{2}10$; $(1100)^{95}1^{3}0$; $(10)^{383}0^{2}$ \\
\hline
$k = 130816$, $r = 256^{130816}$, $s = -256^{131327}$ \\
$0$; $1^{2}0$; $0^{4}10$; $1^{8}0^{2}10$; $0^{16}1^{4}0^{2}10$; $1^{16}0^{24}1^{4}0^{2}10$; $1^{8}0^{8}1^{8}0^{8}1^{8}0^{8}1^{8}0^{8}1^{8}0^{8}1^{12}0^{2}10$; $(11110000)^{23}0^{6}10$; $(1100)^{95}1^{3}0$; $(10)^{383}0^{2}$ \\
\hline
$k = 130816$, $r = 256^{130816}$, $s = -256^{131327}$ \\
$0$; $1^{2}0$; $1^{5}0$; $0^{8}1^{3}0$; $0^{20}1^{3}0$; $1^{16}0^{28}1^{3}0$; $1^{8}0^{8}1^{8}0^{8}1^{8}0^{8}1^{8}0^{8}1^{8}0^{8}1^{8}0^{4}1^{3}0$; $(11110000)^{23}1^{7}0$; $(1100)^{95}0^{2}10$; $(10)^{383}0^{2}$ \\
\hline
$k = 130816$, $r = 256^{130816}$, $s = -256^{131327}$ \\
$0$; $1^{2}0$; $0^{4}10$; $0^{10}10$; $1^{16}0^{6}10$; $1^{16}0^{16}1^{8}0^{6}10$; $1^{8}0^{8}1^{8}0^{8}1^{8}0^{8}1^{8}0^{8}1^{8}0^{22}10$; $(11110000)^{23}1^{4}0^{2}10$; $(1100)^{95}1^{3}0$; $(10)^{383}0^{2}$ \\
\hline
$k = 130816$, $r = 256^{130816}$, $s = -256^{131327}$ \\
$0$; $1^{2}0$; $1^{5}0$; $0^{8}1^{3}0$; $1^{16}0^{4}1^{3}0$; $1^{16}0^{16}1^{8}0^{4}1^{3}0$; $1^{8}0^{8}1^{8}0^{8}1^{8}0^{8}1^{8}0^{8}1^{8}0^{20}1^{3}0$; $(11110000)^{23}1^{7}0$; $(1100)^{95}0^{2}10$; $(10)^{383}0^{2}$ \\
\hline
$k = 130816$, $r = 256^{130816}$, $s = -256^{131327}$ \\
$0$; $1^{2}0$; $0^{4}10$; $1^{8}0^{2}10$; $1^{20}0^{2}10$; $1^{16}0^{16}1^{12}0^{2}10$; $1^{8}0^{8}1^{8}0^{8}1^{8}0^{8}1^{8}0^{8}1^{8}0^{16}1^{4}0^{2}10$; $(11110000)^{23}0^{6}10$; $(1100)^{95}1^{3}0$; $(10)^{383}0^{2}$ \\
\hline
$k = 130816$, $r = 256^{130816}$, $s = -256^{131327}$ \\
$0$; $1^{2}0$; $1^{5}0$; $1^{11}0$; $1^{23}0$; $1^{16}0^{16}1^{15}0$; $1^{8}0^{8}1^{8}0^{8}1^{8}0^{8}1^{8}0^{8}1^{8}0^{16}1^{7}0$; $(11110000)^{23}0^{4}1^{3}0$; $(1100)^{95}0^{2}10$; $(10)^{383}0^{2}$ \\
\hline
$k = 130816$, $r = 256^{130816}$, $s = -256^{131327}$ \\
$0$; $1^{2}0$; $1^{5}0$; $1^{11}0$; $0^{16}1^{7}0$; $1^{16}0^{24}1^{7}0$; $1^{8}0^{8}1^{8}0^{8}1^{8}0^{8}1^{8}0^{8}1^{8}0^{8}1^{15}0$; $(11110000)^{23}0^{4}1^{3}0$; $(1100)^{95}0^{2}10$; $(10)^{383}0^{2}$ \\
\hline
$k = 131328$, $r = 256^{130815}$, $s = -256^{131328}$ \\
$0$; $0^{3}$; $1^{4}0^{2}$; $0^{8}1^{2}0^{2}$; $1^{16}0^{4}1^{2}0^{2}$; $1^{16}0^{16}1^{8}0^{4}1^{2}0^{2}$; $1^{8}0^{8}1^{8}0^{8}1^{8}0^{8}1^{8}0^{8}1^{8}0^{20}1^{2}0^{2}$; $(11110000)^{23}1^{6}0^{2}$; $(1100)^{95}1^{4}$; $(10)^{384}$ \\
\hline
$k = 131328$, $r = 256^{130815}$, $s = -256^{131328}$ \\
$0$; $0^{3}$; $1^{4}0^{2}$; $0^{8}1^{2}0^{2}$; $0^{20}1^{2}0^{2}$; $1^{16}0^{28}1^{2}0^{2}$; $1^{8}0^{8}1^{8}0^{8}1^{8}0^{8}1^{8}0^{8}1^{8}0^{8}1^{8}0^{4}1^{2}0^{2}$; $(11110000)^{23}1^{6}0^{2}$; $(1100)^{95}1^{4}$; $(10)^{384}$ \\
\hline
$k = 131328$, $r = 256^{130815}$, $s = -256^{131328}$ \\
$0$; $0^{3}$; $1^{4}0^{2}$; $1^{10}0^{2}$; $0^{16}1^{6}0^{2}$; $1^{16}0^{24}1^{6}0^{2}$; $1^{8}0^{8}1^{8}0^{8}1^{8}0^{8}1^{8}0^{8}1^{8}0^{8}1^{14}0^{2}$; $(11110000)^{23}0^{4}1^{2}0^{2}$; $(1100)^{95}1^{4}$; $(10)^{384}$ \\
\hline
$k = 131328$, $r = 256^{130815}$, $s = -256^{131328}$ \\
$0$; $0^{3}$; $1^{4}0^{2}$; $1^{10}0^{2}$; $1^{22}0^{2}$; $1^{16}0^{16}1^{14}0^{2}$; $1^{8}0^{8}1^{8}0^{8}1^{8}0^{8}1^{8}0^{8}1^{8}0^{16}1^{6}0^{2}$; $(11110000)^{23}0^{4}1^{2}0^{2}$; $(1100)^{95}1^{4}$; $(10)^{384}$ \\
\hline
$k = 131328$, $r = 256^{130815}$, $s = -256^{131328}$ \\
$0$; $0^{3}$; $0^{6}$; $0^{12}$; $1^{16}0^{8}$; $1^{16}0^{16}1^{8}0^{8}$; $1^{8}0^{8}1^{8}0^{8}1^{8}0^{8}1^{8}0^{8}1^{8}0^{8}1^{16}$; $(11110000)^{24}$; $(1100)^{96}$; $(10)^{384}$ \\
\hline
$k = 131328$, $r = 256^{130815}$, $s = -256^{131328}$ \\
$0$; $0^{3}$; $0^{6}$; $1^{8}0^{4}$; $0^{16}1^{4}0^{4}$; $1^{16}0^{24}1^{4}0^{4}$; $1^{8}0^{8}1^{8}0^{8}1^{8}0^{8}1^{8}0^{8}1^{8}0^{8}1^{12}0^{4}$; $(11110000)^{23}1^{8}$; $(1100)^{96}$; $(10)^{384}$ \\
\hline
$k = 131328$, $r = 256^{130815}$, $s = -256^{131328}$ \\
$0$; $0^{3}$; $0^{6}$; $1^{8}0^{4}$; $1^{20}0^{4}$; $1^{16}0^{16}1^{12}0^{4}$; $1^{8}0^{8}1^{8}0^{8}1^{8}0^{8}1^{8}0^{8}1^{8}0^{16}1^{4}0^{4}$; $(11110000)^{23}1^{8}$; $(1100)^{96}$; $(10)^{384}$ \\
\hline
$k = 131328$, $r = 256^{130815}$, $s = -256^{131328}$ \\
$0$; $0^{3}$; $0^{6}$; $0^{12}$; $0^{24}$; $1^{48}$; $1^{8}0^{8}1^{8}0^{8}1^{8}0^{8}1^{8}0^{8}1^{8}0^{8}1^{8}0^{8}$; $(11110000)^{24}$; $(1100)^{96}$; $(10)^{384}$ \\
\hline
$k = 131328$, $r = 256^{130815}$, $s = -256^{131328}$ \\
$0$; $1^{2}0$; $0^{4}10$; $1^{8}0^{2}10$; $0^{16}1^{4}0^{2}10$; $1^{16}0^{24}1^{4}0^{2}10$; $1^{8}0^{8}1^{8}0^{8}1^{8}0^{8}1^{8}0^{8}1^{8}0^{8}1^{12}0^{2}10$; $(11110000)^{23}0^{6}10$; $(1100)^{95}1^{3}0$; $(10)^{383}1^{2}$ \\
\hline
$k = 131328$, $r = 256^{130815}$, $s = -256^{131328}$ \\
$0$; $1^{2}0$; $0^{4}10$; $0^{10}10$; $0^{22}10$; $1^{16}0^{30}10$; $1^{8}0^{8}1^{8}0^{8}1^{8}0^{8}1^{8}0^{8}1^{8}0^{8}1^{8}0^{6}10$; $(11110000)^{23}1^{4}0^{2}10$; $(1100)^{95}1^{3}0$; $(10)^{383}1^{2}$ \\
\hline
$k = 131328$, $r = 256^{130815}$, $s = -256^{131328}$ \\
$0$; $1^{2}0$; $1^{5}0$; $0^{8}1^{3}0$; $0^{20}1^{3}0$; $1^{16}0^{28}1^{3}0$; $1^{8}0^{8}1^{8}0^{8}1^{8}0^{8}1^{8}0^{8}1^{8}0^{8}1^{8}0^{4}1^{3}0$; $(11110000)^{23}1^{7}0$; $(1100)^{95}0^{2}10$; $(10)^{383}1^{2}$ \\
\hline
$k = 131328$, $r = 256^{130815}$, $s = -256^{131328}$ \\
$0$; $1^{2}0$; $0^{4}10$; $0^{10}10$; $1^{16}0^{6}10$; $1^{16}0^{16}1^{8}0^{6}10$; $1^{8}0^{8}1^{8}0^{8}1^{8}0^{8}1^{8}0^{8}1^{8}0^{22}10$; $(11110000)^{23}1^{4}0^{2}10$; $(1100)^{95}1^{3}0$; $(10)^{383}1^{2}$ \\
\hline
$k = 131328$, $r = 256^{130815}$, $s = -256^{131328}$ \\
$0$; $1^{2}0$; $1^{5}0$; $1^{11}0$; $0^{16}1^{7}0$; $1^{16}0^{24}1^{7}0$; $1^{8}0^{8}1^{8}0^{8}1^{8}0^{8}1^{8}0^{8}1^{8}0^{8}1^{15}0$; $(11110000)^{23}0^{4}1^{3}0$; $(1100)^{95}0^{2}10$; $(10)^{383}1^{2}$ \\
\hline
$k = 131328$, $r = 256^{130815}$, $s = -256^{131328}$ \\
$0$; $1^{2}0$; $1^{5}0$; $0^{8}1^{3}0$; $1^{16}0^{4}1^{3}0$; $1^{16}0^{16}1^{8}0^{4}1^{3}0$; $1^{8}0^{8}1^{8}0^{8}1^{8}0^{8}1^{8}0^{8}1^{8}0^{20}1^{3}0$; $(11110000)^{23}1^{7}0$; $(1100)^{95}0^{2}10$; $(10)^{383}1^{2}$ \\
\hline
$k = 131328$, $r = 256^{130815}$, $s = -256^{131328}$ \\
$0$; $1^{2}0$; $1^{5}0$; $1^{11}0$; $1^{23}0$; $1^{16}0^{16}1^{15}0$; $1^{8}0^{8}1^{8}0^{8}1^{8}0^{8}1^{8}0^{8}1^{8}0^{16}1^{7}0$; $(11110000)^{23}0^{4}1^{3}0$; $(1100)^{95}0^{2}10$; $(10)^{383}1^{2}$ \\
\hline
$k = 131328$, $r = 256^{130815}$, $s = -256^{131328}$ \\
$0$; $1^{2}0$; $0^{4}10$; $1^{8}0^{2}10$; $1^{20}0^{2}10$; $1^{16}0^{16}1^{12}0^{2}10$; $1^{8}0^{8}1^{8}0^{8}1^{8}0^{8}1^{8}0^{8}1^{8}0^{16}1^{4}0^{2}10$; $(11110000)^{23}0^{6}10$; $(1100)^{95}1^{3}0$; $(10)^{383}1^{2}$ \\
\hline
$k = 131838$, $r = 254^{131838}$, $s = -258^{130305}$ \\
$0$; $1^{2}0$; $1010^{3}$; $10^{3}10^{7}$; $10^{7}10^{15}$; $10^{15}10^{15}1^{16}$; $10^{7}1^{16}0^{8}10^{7}1^{16}0^{8}1^{8}0^{8}1^{8}0^{8}$; \\$(1000111111110000111100001111000011110000111100001111000011110000)^{2}(11110000)^{8}$; \\$(101111001100110011001100110011001100110011001100110011001100110011001100110011001$\\$10011001100110011001100110011001100110011001100)^{2}(1100)^{32}$; $1^{3}(01)^{127}1^{2}(01)^{254}0$ \\
\hline
$k = 260610$, $r = 2^{260610}$, $s = -510^{1533}$ \\
$0$; $0^{3}$; $101010$; $1^{3}(0111)^{2}0$; $1^{7}01^{7}01^{7}0$; $1^{15}01^{15}01^{15}0$; $1^{31}01^{31}01^{31}0$; $1^{63}01^{63}01^{63}0$; $1^{127}01^{127}01^{127}0$; $1^{255}01^{255}01^{255}0$ \\
\hline
\caption{The PDS in $C_{2^n}\times C_{2^n}$, $n=9$}
\label{PDSn9}
\end{longtable}

\section{Acknowledgments}

We are grateful to Sam Houston State University (SHSU) and the IT\textsf{@}Sam department at SHSU for their assistance in building and maintaining the server on which we obtained our computational results. 
The second author is grateful for the hospitality of the University of Richmond during several visits to discuss partial difference sets with James A. Davis and John B. Polhill.

\clearpage

\bibliographystyle{plain}
\bibliography{PDSbib}

\begin{thebibliography}{10}

\bibitem{BeinekeWilson2004}
Lowell~W. Beineke and Robin~J. Wilson, editors.
\newblock {\em Topics in algebraic graph theory}, volume 102 of {\em
  Encyclopedia of Mathematics and its Applications}.
\newblock Cambridge University Press, Cambridge, 2004.

\bibitem{Biggs1993}
Norman Biggs.
\newblock {\em Algebraic graph theory}.
\newblock Cambridge Mathematical Library. Cambridge University Press,
  Cambridge, second edition, 1993.

\bibitem{Bose1963}
R.~C. Bose.
\newblock Strongly regular graphs, partial geometries and partially balanced
  designs.
\newblock {\em Pacific J. Math.}, 13:389--419, 1963.

\bibitem{BridgesMena1979}
W.~G. Bridges and R.~A. Mena.
\newblock Rational circulants with rational spectra and cyclic strongly regular
  graphs.
\newblock {\em Ars Combin.}, 8:143--161, 1979.

\bibitem{BCN1989}
A.~E. Brouwer, A.~M. Cohen, and A.~Neumaier.
\newblock {\em Distance-regular graphs}, volume~18 of {\em Ergebnisse der
  Mathematik und ihrer Grenzgebiete (3) [Results in Mathematics and Related
  Areas (3)]}.
\newblock Springer-Verlag, Berlin, 1989.

\bibitem{CameronvanLint1991}
P.~J. Cameron and J.~H. van Lint.
\newblock {\em Designs, graphs, codes and their links}, volume~22 of {\em
  London Mathematical Society Student Texts}.
\newblock Cambridge University Press, Cambridge, 1991.

\bibitem{DavisJedwab1993}
James~A. Davis and Jonathan Jedwab.
\newblock A survey of {H}adamard difference sets.
\newblock In {\em Groups, difference sets, and the {M}onster ({C}olumbus, {OH},
  1993)}, volume~4 of {\em Ohio State Univ. Math. Res. Inst. Publ.}, pages
  145--156. de Gruyter, Berlin, 1996.

\bibitem{StuckeyThesis}
Jessica Ellis.
\newblock {\em Rational Idempotents and Cayley Graphs}.
\newblock (Unpublished), 2010.
\newblock Thesis (M.S.)--Sam Houston State University, Huntsville, TX.

\bibitem{GAP4}
The GAP~Group.
\newblock {\em {GAP -- Groups, Algorithms, and Programming, Version 4.7.9}},
  2015.

\bibitem{GearheartThesis}
Suzi~Vasquez Gearheart.
\newblock {\em On Partial Difference Sets in $C_{16} \times C_{16}$}.
\newblock (Unpublished), 2013.
\newblock Thesis (M.S.)--Sam Houston State University, Huntsville, TX.

\bibitem{GodsilRoyle2001}
Chris Godsil and Gordon Royle.
\newblock {\em Algebraic graph theory}, volume 207 of {\em Graduate Texts in
  Mathematics}.
\newblock Springer-Verlag, New York, 2001.

\bibitem{HaemersSpence}
Willem Haemers and Edward Spence.
\newblock The pseudo-geometric graphs for generalised quadrangles of order
  (3,t).
\newblock {\em European Journal of Combinatorics}, 22:839--845, August 2001.

\bibitem{AutomAbGps}
Christopher~J. Hillar and Darren~L. Rhea.
\newblock Automorphisms of finite abelian groups.
\newblock {\em The American Mathematical Monthly}, 114(10):917--923, 2007.

\bibitem{LeifmanMuzychuk95}
Yefim~I. Leifman and Mikhail~E. Muzychuk.
\newblock Strongly regular {C}ayley graphs over the group {$\mathbb {Z}_{p^n}
  \oplus \mathbb {Z}_{p^n}$}.
\newblock {\em Discrete Math.}, 305(1-3):219--239, 2005.

\bibitem{Ma1984}
S.~L. Ma.
\newblock Partial difference sets.
\newblock {\em Discrete Math.}, 52(1):75--89, 1984.

\bibitem{Ma1994}
S.~L. Ma.
\newblock A survey of partial difference sets.
\newblock {\em Des. Codes Cryptogr.}, 4(3):221--261, 1994.

\bibitem{Ma1997}
S.~L. Ma.
\newblock Some necessary conditions on the parameters of partial difference
  sets.
\newblock {\em J. Statist. Plann. Inference}, 62(1):47--56, 1997.

\bibitem{OurWebsite}
Martin~E.\ Malandro and Ken~W.\ Smith.
\newblock Partial difference sets in ${C}_{2^n} \times {C}_{2^n}$.
\newblock \url{http://www.shsu.edu/mem037/PDSC2nC2n.html}.
\newblock Accessed: August 20, 2018.

\bibitem{Sage}
W.\thinspace{}A. Stein et~al.
\newblock {\em {S}age {M}athematics {S}oftware ({V}ersion 5.6.0)}.
\newblock The Sage Development Team, 2013.
\newblock {\tt http://www.sagemath.org}.

\bibitem{vanLintWilson2001}
J.~H. van Lint and R.~M. Wilson.
\newblock {\em A course in combinatorics}.
\newblock Cambridge University Press, Cambridge, second edition, 2001.

\end{thebibliography}

\end{document}